\documentclass[graybox]{svmult}

\usepackage{mathptmx}       
\usepackage{helvet}         
\usepackage{courier}        
\usepackage{type1cm}        
\usepackage{makeidx}         
\usepackage{graphicx}        
\usepackage{multicol}        
\usepackage[bottom]{footmisc}
\usepackage{bm}

\makeindex             

\usepackage{latexsym}
\usepackage{dsfont}
\usepackage{bbm}
\usepackage{amssymb}
\usepackage{amsmath}

\newtheorem{Theorem}{Theorem}[section]
\newtheorem{Lemma}[Theorem]{Lemma}

\newtheorem{Prop}[Theorem]{Proposition}
\newtheorem{Rem}[Theorem]{Remark}

\def\cA{\mathcal{A}}
\def\cB{\mathcal{B}}
\def\cC{\mathcal{C}}

\def\cO{\mathcal{O}}

\def\fA{\mathfrak{A}}

\def\Erw{\mathbb{E}}

\def\N{\mathbb{N}}
  
\def\Prob{\mathbb{P}} 

\def\R{\mathbb{R}}

\def\X{\mathbb{X}}

\def\Z{\mathbb{Z}}

\def\eps{\varepsilon}

\def\vth{\vartheta}

\def\1{\vec{1}}
\def\3{{\ss}}

\def\eqdist{\stackrel{d}{=}}

\def\wh{\widehat}

\def\ovl{\overline}
\def\uline{\underline}

\def\sign{\textsl{sign}}

\allowdisplaybreaks[4] 

\begin{document}

\title*{On the stationary tail index of iterated random Lipschitz functions}
\titlerunning{On the stationary tail index of iterated random functions}
\author{Gerold Alsmeyer}
\institute{Gerold Alsmeyer \at Inst.~Math.~Statistics, Department
of Mathematics and Computer Science, University of M\"unster,
Orl\'eans-Ring 10, D-48149 M\"unster, Germany.\at
\email{gerolda@math.uni-muenster.de}\\
Research supported by the Deutsche Forschungsgemeinschaft (SFB 878).}

\maketitle

\abstract{Let $\Psi,\Psi_{1},\Psi_{2},...$ be a sequence of iid random Lipschitz maps from a complete separable metric space $(\X,d)$ with unbounded metric $d$ to itself and let $X_{n}=\Psi_{n}\circ...\circ\Psi_{1}(X_{0})$ for $n=1,2,...$ be the associated Markov chain of forward iterations with inital value $X_{0}$ which is independent of the $\Psi_{n}$. Provided that $(X_{n})_{n\ge 0}$ has a stationary law $\pi$ and picking an arbitrary reference point $x_{0}\in\X$, we will study the tail behavior of $d(x_{0},X_{0})$ under $\Prob_{\pi}$, viz. the behavior of $\Prob_{\pi}(d(x_{0},X_{0})>t)$ as $t\to\infty$, in cases when there exist (relatively simple) nondecreasing continuous random functions $F,G:\R_{\ge}\to\R_{\ge}$ such that
$$ F(d(x_{0},x))\ \le\ d(x_{0},\Psi(x))\ \le\ G(d(x_{0},x)) $$
for all $x\in\X$ and $n\ge 1$. In a nutshell, our main result states that, if the iterations of iid copies of $F$ and $G$ constitute contractive iterated function systems with unique stationary laws $\pi_{F}$ and $\pi_{G}$ having power tails of order $\vth_{F}$ and $\vth_{G}$ at infinity, respectively, then lower and upper tail index of $\nu=\Prob_{\pi}(d(x_{0},X_{0})\in\cdot)$ (to be defined in Section \ref{sec:tail index}) are falling in $[\vth_{G},\vth_{F}]$. If $\vth_{F}=\vth_{G}$, which is the most interesting case, this leads to the exact tail index of $\nu$. We illustrate our method, which may be viewed as a supplement of Goldie's implicit renewal theory, by a number of popular examples including the AR(1)-model with ARCH errors and random logistic transforms.}
\bigskip



\section{Introduction}

Iterations of iid random Lipschitz functions on $\X\subset\R$ constitute an interesting class of recursive Markov chains which arise in various fields like queuing theory, population dynamics or mathematical finance. If the considered chain has a nondegenerate (not necessarily unique) stationary law and unbounded state space $\X$, it is natural to ask about the tail behavior of this law at the remote ends of $\X$. An answer can often be obtained with the help of Goldie's \cite{Goldie:91} implicit renewal theory when the random Lipschitz function is approximately linear at these ends and some additional conditions hold true. The method to be introduced in this article may be viewed, in the first place, as a supplement to Goldie's approach by making it sometimes easier to verify his conditions in concrete examples (see Subsection \ref{subsec:AR-ARCH}), and
also as an extension by being applicable to the more general situation when $\X$ is an arbitrary metric space, thus particularly including $\X=\R^{m}$ for any $m\ge 2$. In order to be more precise, we first need to describe our general setup.

\vspace{.2cm}
Let $(\X,d)$ be a complete separable metric space with Borel-$\sigma$-field $\cB(\X)$ and unbounded metric $d$. A temporally homogeneous Markov chain $(X_{n})_{n\ge 0}$ with state space $\X$ is called \emph{iterated function system (IFS) of iid Lipschitz maps}\index{iterated function system} if it satisfies a recursion of the form
\begin{equation}\label{eq:recursive eq for Mn}
X_{n}=\Psi(\theta_{n},X_{n-1})
\end{equation}
for $n\ge 1$, where
\begin{description}[(IFS-3)]
\item[(IFS-1)] $X_{0},\theta_{1},\theta_{2},...$ are independent random elements
on a common probability space $(\Omega,\fA,\Prob)$;
\item[(IFS-2)] $\theta_{1},\theta_{2},...$ are identically distributed with
common distribution $\Lambda$ and taking values in a measurable
space $(\Theta,\cA)$;
\item[(IFS-3)] $\Psi:(\Theta\times\X,\mathcal{A}\otimes\cB(\X))\to (\X,\cB(\X))$ is
jointly measurable and Lipschitz continuous\index{Lipschitz!continuous} in the second argument, that is
$$ d(\Psi(\theta,x),\Psi(\theta,y))\le C_{\theta}\,d(x,y) $$
for all $x,y\in\X$, $\theta\in\Theta$ and a suitable $C_{\theta}\in\R_{\ge}$.
\end{description}

A natural way to generate an IFS is to first pick an iid sequence $\Psi_{1},\Psi_{2},...$ of random elements from the space $\cC_{Lip}(\X)$ of Lipschitz self-maps on $\X$ and to then produce a Markov chain $(X_{n})_{n\ge 0}$ by picking an initial value $X_{0}$ and defining
\begin{equation}\label{eq:standard form of IFS}
X_{n}\ :=\ \Psi_{n:1}(X_{0})
\end{equation}
for each $n\ge 1$, where $\Psi_{n:1}:=\Psi_{n}\circ...\circ\Psi_{1}$. In the context of the above definition, $\Psi_{n}=\Psi(\theta_{n},\cdot)$, but it becomes a measurable object only if we endow $\cC_{Lip}(\X)$ with a suitable $\sigma$-field. Further defining the Lipschitz constant of $\psi\in\cC_{Lip}(\X)$ as
\begin{equation}\label{def:Lipschitz constant}\index{Lipschitz!constant}
L(\psi)\ :=\ \sup_{x,y\in\X,\,x\ne y}\frac{d(\psi(x),\psi(y))}{d(x,y)},
\end{equation}
the mappings $\psi\mapsto L(\psi)$ and $(\psi,x)\mapsto\psi(x)$ are then Borel functions on $\cC_{Lip}(\X)$ and $\cC_{Lip}(\X)\times\X$, respectively. For details regarding these facts see the excellent survey by Diaconis and Freedman \cite[Section 5]{DiaconisFr:99}.

Closely related to the \emph{forward iterations} $X_{n}=\Psi_{n:1}(X_{0})$ are the \emph{backward iterations}
$$ \wh{X}_{n}\ :=\ \Psi_{1:n}(X_{0}),\quad\Psi_{1:n}\ :=\ \Psi_{1}\circ...\circ\Psi_{n}, $$
for $n\ge 1$, the obvious connection being $X_{n}\eqdist\wh{X}_{n}$ for all $n\ge 0$ $(\wh{X}_{0}:=X_{0})$. On the other hand, the pathwise behavior of forward and backward iterations differs drastically. Suppose $\Erw\log^{+}L(\Psi)<\infty$ and the \emph{jump-size condition}
\begin{equation}\label{cond:jump-size}
\Erw\log^{+}d(x_{0},\Psi(x_{0}))\ <\ \infty\quad\text{for some (and then all) }x_{0}\in\X.
\end{equation}
Elton \cite{Elton:90} then showed that if the IFS is \emph{contractive} in the sense that
\begin{equation}\label{cond:ultimate contractive}
\log l(\Psi)\ :=\ \lim_{n\to\infty}\frac{1}{n}\log L(\Psi_{n:1})\ <\ 0\quad\text{ a.s.},
\end{equation}
(the a.s. convergence being ensured by the subadditive ergodic theorem) or, a fortiori,
\emph{mean contractive}, i.e.
\begin{equation}\label{cond:mean contraction}
\Erw\log L(\Psi)\ <\ 0,
\end{equation}
then:
\begin{description}[(c)]
\item[(a)] the forward iteration $X_{n}$ converges weakly to a random variable $X_{\infty}$ with law $\pi$ under each $\Prob_{x}:=\Prob(\cdot|X_{0}=x)$, $x\in\X$;
\item[(b)] the backward iteration $\wh{X}_{n}$ converges $\Prob_{x}$-a.s. to some $\wh{X}_{\infty}$ with law $\pi$;
\item[(c)] $\pi$ is the \emph{unique} stationary distribution of the Markov chain $(X_{n})_{n\ge 0}$ and the latter an ergodic sequence under $\Prob_{\pi}$.
\end{description}
Moreover, 
\begin{description}[(c)]
\item[(d)] the stochastic fixed-point equation (SFPE) $X_{0}\eqdist\Psi(X_{0})$ holds true  under $\Prob_{\pi}$. 
\end{description}
While Elton actually stated his result for general stationary sequences $(\Psi_{n})_{n\ge 1}$, proofs for the iid case including convergence rate results may also be found in the afore-mentioned survey \cite{DiaconisFr:99} and in \cite{AlsFuh:01,AlsHolk:09}.

\vspace{.1cm}
Being interested in the tail behavior of a stationary law of an IFS, the existence of such a law must naturally be guaranteed for our analysis. On the other hand, this does not necessarily require the IFS to be contractive. For instance, if the backward iteration $\wh{X}_{n}$ converges $\Prob_{x}$-a.s. to a limit $\wh{X}_{\infty}$ with law $\pi$ which does not depend on $x$ (statement (b) above), then statements (a), (c) and (d) are true as well without further ado. This is a result due to Letac \cite{Letac:86} and often called Letac's principle. It holds true for any sequence of iid continuous, but not necessarily Lipschitz functions $\Psi_{1},\Psi_{2},...$ Non-contractive IFS with nondegenerate stationary laws may also be found in the class of iterations of iid piecewise monotone, continuous and uniformly expanding self-maps of the unit interval, see the monography by Boyarsky and G\'ora \cite{BoyGora:97} and the references therein. IFS of iid random Lipschitz maps which are contractive only on a subset of their domain are quite popular in population dynamics. Two prominent examples, viz. random logistic transforms and the stochastic Ricker model, will also be studied in this article, see Subsections \ref{subsec:logistic} and \ref{subsec:Ricker}. In view of these remarks we wish to point out that the method to be introduced here is not restricted to the framework in which Elton's result is stated. We proceed to a quick outline of the idea on which it is based. 

\vspace{.1cm}
So let $(X_{n})_{n\ge 0}$ be any IFS of iid Lipschitz maps with generic copy $\Psi$ and stationary law $\pi$ (not necessarily unique). If $\Prob_{\pi}(d(x_{0},\Psi(X_{0}))\le r)<1$ for all $r>0$, $x_{0}\in\X$ an arbitrary reference point, it is natural to ask for more detailed information about the tail behavior of $Q=\Prob_{\pi}(d(x_{0},X_{0})\in\cdot)$. Focussing on situations when $Q$ is heavy-tailed, the main contribution of this article is to show that this may be accomplished by finding bounds for $d(x_{0},X_{n})=d(x_{0},\Psi_{n:1}(X_{0}))$ in terms of relatively \emph{simple} contractive IFS on $\R_{\ge}$ (which does not mean that $(X_{n})_{n\ge 0}$ itself is contractive!). Lemma \ref{lem:basic lemma} constitutes the basic result to embark on. When $\X=\R$ with the usual Euclidean metric, it is natural to take $x_{0}=0$, thus asking for the tail behavior of $|X_{0}|$ under $\Prob_{\pi}$ and thus of $\pi$ itself. 
In this case one may further distinguish between the tail behavior of $X_{0}$ at $+\infty$ and $-\infty$. Goldie's implicit renewal theorem, to be shortly reviewed in Subsection \ref{subsec:IRT}, will be a helpful ingredient to our analysis because it can be used to find the tail behavior of the afore-mentioned bounding simple IFS of iid random Lipschitz functions on $\R$. Not surprisingly, this will require further assumptions beyond those stated above for Elton's result. 

Let us finally mention that the basic idea of bounding an IFS (or a metric functional thereof) by simpler ones has been utilized earlier, though in a different manner, by Mirek \cite{Mirek:11} in the analysis of iterations of iid contractive Lipschitz maps on $\R^{d}$ and their Birkhoff sums, by Brofferio and Buraczewski \cite{BroBura:13}, who study unbounded invariant measures of such iterations in the critical case, and also by Collamore and Vidyashankar \cite{CollVidya:13} (see their cancellation condition after Thm. 2).

\section{Tail index}\label{sec:tail index}

Let $X$ be a random variable on a probability space $(\Omega,\fA,\Prob)$ with distribution $\Lambda$. Then $X$ and $\Lambda$ are said to have
\begin{itemize}
\item \emph{lower tail index $\vth_{*}$ (at $+\infty$)} if $\displaystyle{\limsup_{x\to\infty}\frac{\log\Prob(X>x)}{\log x}=-\vth_{*}}<0$.\vspace{.1cm}
\item \emph{upper tail index} $\vth^{*}$ if $\displaystyle{\liminf_{x\to\infty}\frac{\log\Prob(X>x)}{\log x}=-\vth^{*}}>-\infty$.\vspace{.1cm}
\item \emph{tail index} $\vth$ if $\vth_{*}=\vth^{*}=\vth\in\R_{>}$.\vspace{.1cm}
\item \emph{exact tail index} $\vth\in\R_{>}$ if $\displaystyle{0<\liminf_{x\to\infty}x^{\vth}\Prob(X>x)\le\limsup_{x\to\infty}x^{\vth}\Prob(X>x)<\infty}$.
\end{itemize}
As one can readily see, the lower tail index of $X$, if it exists, is given by the maximal positive $\vth_{*}$ such that
\begin{align*}
\lim_{x\to\infty}x^{\vth_{*}-\eps}\,\Prob(X>x)=0\quad\text{for all }\eps>0,
\end{align*}
while the upper tail index equals the minimal positive $\vth^{*}$ such that
\begin{align*}
\lim_{x\to\infty}x^{\vth^{*}+\eps}\,\Prob(X>x)=\infty\quad\text{for all }\eps>0.
\end{align*}
Hence, if $\vth_{*},\vth^{*}$ both exist, then
\begin{align}\label{eq:upper and lower TI}
\frac{1}{x^{\vth^{*}+\eps}}\ \le\ \Prob(X>x)\ \le\ \frac{1}{x^{\vth_{*}-\eps}}\quad\text{for all }\eps>0\text{ and $x$ sufficiently large}.
\end{align}
Regarding exactness of a tail index, a stronger definition than the one stated above is that
\begin{equation}\label{eq:very exact TI}
\lim_{x\to\infty}x^{\vth}\Prob(X>x)\quad\text{exists and is finite and positive}.
\end{equation}
This stronger form actually holds in many examples including those discussed in this article.

The existence of lower and upper tail index is a rather weak property in the sense that, generally, it does not provide much information about the asymptotic behavior, as $x\to\infty$, of the ratio
$$ \frac{\ovl{\Lambda}(tx)}{\ovl{\Lambda}(x)}\quad\text{ for }t\in [1,T],\ T>1., $$
where $\ovl{\Lambda}(x):=\Prob(X>x)$. If, for each $T>1$ and uniformly in $t\in [1,T]$,
\begin{equation}\label{eq:Karamata, Matuszewska}
c(1+o(1))t^{a}\ \le\ \frac{\ovl{\Lambda}(tx)}{\ovl{\Lambda}(x)}\ \le\ c'(1+o(1))t^{a'}\quad (x\to\infty),
\end{equation}
then the lower and upper \emph{Karamata index} of $\ovl{\Lambda}$ are defined as the
the supremum $\alpha_{*}\ge -\infty$ over all $a$ and the infimum $\alpha^{*}\le 0$ over all $a'$, respectively, such that \eqref{eq:Karamata, Matuszewska} holds with $c=c'=1$. Similarly, the lower and upper \emph{Matuszewska index} of $\ovl{\Lambda}$ are defined as the supremum $\beta_{*}$ over $a$ and infimum $\beta^{*}$ over $a'$, respectively, such that \eqref{eq:Karamata, Matuszewska} holds with suitable $c=c(a),c'=c'(a')\in\R_{>}$, see Bingham et al. \cite[Section 2.1]{BingGolTeug:89}. Obviously,
$$ \alpha_{*}\ \le\ \beta_{*}\ \le\ \beta^{*}\ \le\ \alpha^{*}. $$
Now observe that \eqref{eq:upper and lower TI} implies
$$ \frac{1}{t^{\vth^{*}+\eps}\,x^{\vth^{*}-\vth_{*}+2\eps}}\ \le\ \frac{\ovl{\Lambda}(tx)}{\ovl{\Lambda}(x)}\ \le\ \frac{x^{\vth^{*}-\vth_{*}+2\eps}}{t^{\vth_{*}-\eps}} $$
and hence provides no information about the afore-mentioned indices if these are nontrivial (in $(-\infty,0)$). On the other hand, one can easily check that, if $X$ has exact tail index $\vth$, then $\vth=\beta_{*}=\beta^{*}$, and if a fortiori \eqref{eq:very exact TI} holds, then even $\vth=\alpha_{*}=\alpha^{*}$.

\section{A basic lemma}\label{sec:basic lemma}

Given an IFS with stationary law $\pi$ as defined above, fix any reference point $x_{0}\in\X$ and consider the random variables $D_{n:1}(x):=d(x_{0},\Psi_{n:1}(x))$ and $D_{1:n}(x):=d(x_{0},\Psi_{1:n}(x))$ for $n\in\N_{0}$ and $x\in\X$.

Our goal is to provide conditions for the existence of both, lower and upper tail index of $\Lambda:=\Prob_{\pi}(d(x_{0},X_{0})\in\cdot)$ and thus of $D_{n:1}(X_{0}),\,D_{1:n}(X_{0})$ for all $n\in\N_{0}$, when $X_{0}$ denotes a random variable with law $\pi$ and independent of $\Psi_{1},\Psi_{2},...$ The basic ingredient is the following ``sandwich lemma'' which holds true for arbitrary sequences of random functions $\Psi_{1},\Psi_{2},...:\X\to\X$.

\begin{Lemma}\label{lem:basic lemma}
Suppose there exist nondecreasing and continuous random functions $F_{n},G_{n}:I\to I$ for $n\ge 1$ such that $\R_{\ge}\subset I$ and, for some $x_{0}\in\X$,
\begin{description}[(C2)]
\item[(C1)] $(\Psi_{n},F_{n},G_{n})$ are iid for $n\ge 1$ and independent of $X_{0}$.
\item[(C2)] $F_{n}(d(x_{0},x))\le d(x_{0},\Psi_{n}(x))\le G_{n}(d(x_{0},x))$ a.s.\ for all $x\in\X$ and $n\ge 1$.
\end{description}
Then
\begin{align*}
F_{n:1}(d(x_{0},x))\ &\le\ D_{n:1}(x)\ \le\ G_{n:1}(d(x_{0},x))\\
\text{and}\quad F_{1:n}(d(x_{0},x))\ &\le\ D_{1:n}(x)\ \le\ G_{1:n}(d(x_{0},x))
\end{align*}
holds true a.s. for all $x\in\X$ and $n\ge 1$.
\end{Lemma}

\begin{proof}
Since $D_{n:1}(x)=d(x_{0},\Psi_{n}(\Psi_{n-1:1}(x)))$ for each $x\in\X$ and $n\ge 1$, we obtain by repeated use of (C2) in combination with the monotonicity of the $F_{n},G_{n}$
\begin{align*}
F_{n:1}(d(x_{0},x))\ &\le\ F_{n:2}(d(x_{0},\Psi_{1}(x)))\ =\ F_{n:2}(D_{1:1}(x))\\
&\le\ F_{n:3}(d(x_{0},\Psi_{2:1}(x)))\ =\ F_{n:3}(D_{2:1}(x))\\
&\quad\vdots\\
&\le\ F_{n}(d(x_{0},\Psi_{n-1:1}(x)))\ =\ F_{n}(D_{n-1:1}(x))\\
&\le\ d(x_{0},\Psi_{n:1}(x))\ =\ D_{n:1}(x)\\
&\le\ G_{n}(d(x_{0},\Psi_{n-1:1}(x)))\ =\ G_{n}(D_{n-1:1}(x))\\
&\quad\vdots\\
&\le\ G_{n:2}(d(x_{0},\Psi_{1}(x)))\ =\ G_{n:2}(D_{1:1}(x))\\
&\le\ G_{n:1}(d(x_{0},x))\quad\text{a.s.}
\end{align*}
for all $n\ge 1$.\qed
\end{proof}

\begin{Rem}\label{rem:basic lemma}\rm
We note that condition (C2) above for some $x_{0}\in\X$ implies the very same condition for any other reference point $x_{1}\in\X$. Indeed,
\begin{equation*}
\wh{F}_{n}(d(x_{1},x))\ \le\ d(x_{1},\Psi_{n}(x))\ \le\ \wh{G}_{n}(d(x_{1},x))\quad\text{a.s.}
\end{equation*}
for all $n\ge 1$ and $x\in\X$, where
$$ \wh{F}_{n}(t)\ :=\ F_{n}(t)-d(x_{0},x_{1})\quad\text{and}\quad\wh{G}_{n}(t)\ :=\ G_{n}(t)+d(x_{0},x_{1}) $$
are again nondecreasing and continuous functions. This follows directly by a simple application of the triangular inequality, viz.
$$ d(x_{0},\Psi_{n}(x))-d(x_{0},x_{1})\ \le\ d(x_{1},\Psi_{n}(x))\ \le\ d(x_{0},\Psi_{n}(x))+d(x_{0},x_{1}). $$
Note also the that $L(F)=L(\wh{F})$ and $L(G)=L(\wh{G})$.
\end{Rem}

\begin{Rem}\label{rem2:basic lemma}\rm
It should be clear that the lower and upper estimates for $D_{n:1}(x)$ and $D_{1:n}(x)$ in Lemma \ref{lem:basic lemma} hold independently in the sense that the lower estimate depends only on $F_{n}$, while the second one depends only on $G_{n}$.
\end{Rem}

\begin{Rem}\label{rem3:basic lemma}\rm
In the situation of Lemma \ref{lem:basic lemma}, let us further assume that 
\begin{description}[(b)]
\item[(a)] the $\Psi_{n}$ are continuous so that $(X_{n})_{n\ge 0}$ is a Feller chain,
\item[(b)] the IFS generated by the $G_{n}$ is contractive and
\item[(c)] the Heine-Borel property, viz. the closed balls $\ovl{B}(x,r):=\{y:d(x,y)\le r\}$, $x\in\X$ and $r>0$, are compact subsets of $\X$ (which is clearly true if $\X=\R^{m}$ with the usual topology).
\end{description}
Then $(X_{n})_{n\ge 0}$ possesses at least one stationary distribution.

\vspace{.08cm}\noindent
\emph{Proof.} Note first that (b) ensures the tightness of the sequence $(\Prob(G_{n:1}(x)\in\cdot))_{n\ge 1}$ for any fixed $x\in\X$. As a consequence, the sequence
$$ P^{n}(x,\cdot)\ :=\ \Prob(\Psi_{n:1}(x)\in\cdot),\quad n\ge 1 $$
is also tight because, by Lemma \ref{lem:basic lemma},
\begin{align*}
\Prob(\Psi_{n:1}(x)\not\in\ovl{B}(x_{0},r))\ &=\ \Prob(D_{n:1}(x)>r)\ \le\ \Prob(G_{n:1}(d(x_{0},x))>r)
\end{align*}
for all $r>0$ and $n\in\N$. Finally, the latter implies that $(n^{-1}\sum_{k=1}^{n}P^{k}(x,\cdot))_{n\ge 1}$ contains a weakly convergent subsequence whose limit, by (a), forms a stationary distribution of $(X_{n})_{n\ge 0}$.
\end{Rem}

\section{Implicit renewal theory}\label{sec:IRT}

This section is devoted to a brief review of Goldie's implicit renewal theorem \cite{Goldie:91} and its application to two simple examples that will later be useful in our analysis.

\subsection{Review of Goldie's main results}\label{subsec:IRT}

The following proposition is a condensed version of Goldie's main results \cite[Thm. 2.3 and Cor. 2.4]{Goldie:91}. The connection with stationary laws of IFS is owing to the fact that any such law forms a solution to an SFPE of the form
\begin{equation}\label{eq:SFPE IFS}
X\eqdist\Psi(X)
\end{equation}
for some continuous random function $\Psi$.

\begin{Prop}\label{prop:IRT}\index{theorem!implicit renewal}
{\bf [Implicit renewal theorem]} Let $(\Omega,\fA,\Prob)$ be any probability space, $\Psi:\Omega\times\R\to\R$ a product-measurable function and $M,X$ further random variables on $\Omega$ such that $X$ and $(\Psi,M)$ are independent. Further assume that, for some $\kappa>0$,
\begin{description}[(IRT-2)]
\item[(IRT-1)] $\Erw|M|^{\kappa}=1$.
\item[(IRT-2)] $\Erw|M|^{\kappa}\log^{+}|M|<\infty$.
\item[(IRT-3)] The conditional law $\Prob(\log|M|\in\cdot|M\ne 0)$ of $\log|M|$ given $M\ne 0$ is nonarithmetic, in particular, $\Prob(|M|=1)<1$.
\end{description}
Then $-\infty\le\Erw\log|M|<0$, $0<\mu_{\kappa}:=\Erw|M|^{\kappa}\log |M|<\infty$, and the following assertions hold true:
\begin{description}[(b)]
\item[(a)] Suppose $M$ is a.s.\ nonnegative. If
\begin{equation}\label{eq:IRT cond1-psi}
\Erw\big|(\Psi(X)^{+})^{\kappa}-((MX)^{+})^{\kappa}\big|<\infty
\end{equation}
or, respectively,
\begin{equation}\label{eq:IRT cond2-psi}
\Erw\big|(\Psi(X)^{-})^{\kappa}-((MX)^{-})^{\kappa}\big|<\infty
\end{equation}
then
\begin{equation}\label{right tail estimate X}
\lim_{t\to\infty}t^{\kappa}\,\Prob(X>t)\ =\ C_{+},
\end{equation}
respectively
\begin{equation}\label{left tail estimate X}
\lim_{t\to\infty}t^{\kappa}\,\Prob(X<-t)\ =\ C_{-},
\end{equation}
where $C_{+}$ and $C_{-}$ are given by the equations
\begin{align}
C_{+}\ &=\ \frac{1}{\kappa\mu_{\kappa}}\,\Erw\Big((\Psi(X)^{+})^{\kappa}-((MX)^{+})^{\kappa}\Big),\label{eq:Cplus-psi}\\
C_{-}\ &=\ \frac{1}{\kappa\mu_{\kappa}}\,\Erw\Big((\Psi(X)^{-})^{\kappa}-((MX)^{-})^{\kappa}\Big).\label{eq:Cminus-psi}
\end{align}
\item[(b)] If $\Prob(M<0)>0$ and \eqref{eq:IRT cond1-psi}, \eqref{eq:IRT cond2-psi} are both satisfied, then \eqref{right tail estimate X} and \eqref{left tail estimate X} hold with $C_{+}=C_{-}=C/2$, where
\begin{equation}\label{eq:Cplus=Cminus=C/2-psi}
C\ =\ \frac{1}{\kappa\mu_{\kappa}}\,\Erw\Big(|\Psi(X)|^{\kappa}-|MX|^{\kappa}\Big).
\end{equation}
\end{description}
\end{Prop}

\subsection{Random affine recursions and perpetuities}\label{subsec:perp}

Random affine recursions on $\R$, also called random difference equations, are among the most important and at same time most extensively studied examples of IFS to which Goldie's theory applies successfully.
So let $(M_{n},Q_{n})_{n\ge 1}$ be a sequence of iid $\R^{2}$-valued random vectors
with generic copy $(M,Q)$ and consider the IFS generated by $\Psi_{n}(x)=M_{n}x+Q_{n}$
for $n\ge 1$. Put $\Pi_{0}:=1$ and $\Pi_{n}=M_{1}\cdot...\cdot M_{n}$ for $n\ge 1$.
Existence and uniqueness of a stationary distribution and thus solution to the SFPE 
\begin{equation}\label{eq:SFPE RDE again}
X\ \eqdist\ MX+Q,
\end{equation}
which is given by the law of the so-called \emph{perpetuity}
\begin{equation}\label{eq:def perp}
X\ =\ \sum_{n\ge 1}\Pi_{n-1}Q_{n}
\end{equation}
and obtained as the a.s. limit of the backward iterations, were studied by Vervaat \cite{Vervaat:79} (see also Grincevi\v{c}ius \cite{Grincev:81}) and later by Goldie and Maller \cite{GolMal:00}. The following tail result is due to Kesten \cite[Thm. 5]{Kesten:73}, the form of the constants provided by Goldie \cite[Thm. 4.1]{Goldie:91}.

\begin{Prop}\label{prop:tails of perpetuities}
Suppose that $M$ satisfies (IRT-1)-(IRT-3) and that $\Erw|Q|^{\kappa}<\infty$.
Then there exists a unique solution to the SFPE \eqref{eq:SFPE RDE again}, given by the law of the perpetuity in \eqref{eq:def perp}. This law satisfies \eqref{right tail estimate X} as well as \eqref{left tail estimate X}, where
\begin{equation*}
C_{\pm}=\frac{\Erw\big(((MX+Q)^{\pm})^{\kappa}-((MX)^{\pm})^{\kappa}\big)}{\kappa\mu_{\kappa}}
\end{equation*}
if $M\ge 0$ a.s., while
\begin{equation*}
C_{+}=C_{-}=\frac{\Erw\big(|MX+Q|^{\kappa}-|MX|^{\kappa}\big)}{2\kappa\mu_{\kappa}}
\end{equation*}
if $\Prob(M<0)>0$. Furthermore, $C_{+}+C_{-}>0$ iff
\begin{equation}\label{eq:nondegeneracy RDE}
\Prob(Mc+Q=c)<1\quad\text{for all }c\in\R.
\end{equation}
Finally $\Erw|X|^{p}<\infty$ for all $p\in (0,\kappa)$.
\end{Prop}

\begin{Rem}\label{rem:tails of perpetuities}\rm
Let us point out that, if $M,Q$ and thus $X$ are nonnegative in the previous result, then $C_{-}=0$ and
$$ C_{+}=\frac{\Erw\big((MX+Q)^{\kappa}-(MX)^{\kappa}\big)}{\kappa\mu_{\kappa}} $$
is positive iff $\Prob(Q>0)>0$. An extension to the case when $Q$ may also be negative is provided by the following results that is part of a more general one due to Guivarc'h and Le Page \cite{GuivarchLePage:13}.
\end{Rem}

\begin{Prop}\label{prop:Guivarch perpetuity}
Given the assumptions of Prop.~\ref{prop:tails of perpetuities}, suppose further $\Prob(M>0)=1$ and \eqref{eq:nondegeneracy RDE}. Then $C_{+}$ is positive iff $\Psi(x)=Mx+Q$ possesses no a.s. invariant half-line $(-\infty,c]$, i.e.
\begin{equation}\label{eq:no invariante halfline RDE}
\Prob(Mc+Q>c)>0\quad\text{for all }c\in\R.
\end{equation}
\end{Prop}

Notice that \eqref{eq:no invariante halfline RDE} is particularly fulfilled if
\begin{equation}\label{eq2:no invariante halfline RDE}
\Prob(M\le 1,Q>0)\wedge\Prob(M\ge 1,Q>0)>0.
\end{equation}

\subsection{A variation of the exponential Lindley equation}\label{subsec:Lindley}

Given $r\ge 0$ and iid nonnegative random vectors $(M_{1},Q_{1}),(M_{2},Q_{2}),...$ with generic copy $(M,Q)$, consider the IFS on $\X=\R_{\ge}$ generated by the random Lipschitz functions $\Psi_{n}(x):=Q_{n}\vee (M_{n}\,x\,\1_{(r,\infty)}(x))$, $n\ge 1$. Let $\Pi_{n}$ be defined as in the previous subsection.
Provided that a unique stationary distribution $\pi$ exists, it is given by the law of
\begin{equation}\label{eq:solution mod exp Lindley equation}
X\ :=\ Q_{1}\vee\bigvee_{n\ge 2}\Pi_{n-1}Q_{n}\,\1_{\{Q_{n}>r,\,M_{n-1}Q_{n}>r,\,...,\,M_{2}\cdot...\cdot\,M_{n-1}Q_{n}>r\}}
\end{equation}
and a solution to the SFPE
\begin{equation}\label{eq:mod exp Lindley equation}
X\ \eqdist\ Q\vee (MX\,\1_{(r,\infty)}(X)).
\end{equation}
In the case $r=0$ and $Q=1$ a.s., \eqref{eq:mod exp Lindley equation} equals the exponential version of Lindley's equation
$Y\eqdist (\xi+Y)^{+}$ which is well-known from queueing theory, see e.g. \cite[p. 92ff]{Asmussen:03}.

\begin{Prop}\label{prop:tails of max-type SFPE}
Suppose $M$ satisfies (IRT-1)-(IRT-3) and $\Erw Q^{\kappa}<\infty$.
Then there exists a unique solution to the SFPE \eqref{eq:mod exp Lindley equation}, given by the law of $X$ in \eqref{eq:solution mod exp Lindley equation}. This law satisfies \eqref{right tail estimate X}
with
\begin{equation}\label{eq:Cplus max-type SFPE}
C_{+}=\frac{\Erw\big(((MX\,\1_{(r,\infty)}(X))\vee Q)^{\kappa}-(MX)^{\kappa}\big)}{\kappa\mu_{\kappa}}.
\end{equation}
Moreover, $C_{+}$ is positive iff $\Prob(Q>r)>0$.
\end{Prop}

\begin{proof}
Under the stated assumptions, the given IFS is easily seen to be mean contractive and to satisfy the jump-size condition \eqref{cond:jump-size} (with $x_{0}=0$ and $d(x,y)=|x-y|$).
Hence it possesses a unique stationary distribution obtained by the law of the a.s. limit of the backward iterations which in turn equals $X$ defined by \eqref{eq:solution mod exp Lindley equation}. \eqref{right tail estimate X} with $C_{+}$ given by \eqref{eq:Cplus max-type SFPE} is now directly inferred from Prop.~\ref{prop:IRT} because
\begin{align*}
\Erw&\big|((MX\1_{(r,\infty)}(X))\vee Q)^{\kappa}-(MX)^{\kappa}\big|\\
&=\ \Erw\big|((MX\1_{(r,\infty)}(X))\vee Q)^{\kappa}-(MX)^{\kappa}\big|\,\1_{\{X<r\}\cup\{MX\le Q\}}\\
&\le\ \Erw(MX)^{\kappa}\1_{\{X<r,\,MX>Q\}}+\Erw\big|Q^{\kappa}-(MX)^{\kappa}\big|\,\1_{\{MX\le Q\}}\\
&\le\ r^{\kappa}+\Erw Q^{\kappa}\ <\ \infty
\end{align*}
[which is \eqref{eq:IRT cond1-psi} in that proposition] holds true.

\vspace{.1cm}
Turning to the asserted equivalence, one implication is trivial, for $Q\le r$ a.s.\ entails $X\eqdist Q$ and thus $C_{+}=0$. Hence, suppose $\Prob(Q>r)>0$ and define the predictable first passage time
$$ \tau(t)\ :=\ \inf\{n\ge 1:\Pi_{n-1}>t/r\}\ =\ \inf\{n\ge 1:S_{n-1}>\log(t/r)\},\quad t\ge 0, $$
where $S_{n}:=\log\Pi_{n}$ for $n\ge 0$. The latter sequence forms an ordinary random walk taking values in $\R\cup\{-\infty\}$ and with $\Erw e^{\kappa S_{1}}=\Erw M^{\kappa}=1$ by (IRT-1). Now
\begin{equation}\label{eq:tau(t) finite}
\Prob\left(\sup_{n\ge 1}\Pi_{n-1}>\frac{t}{r}\right)\ =\ \Prob(\tau(t)<\infty),
\end{equation}
and we claim that
\begin{equation}\label{eq:max-type SFPE first passage estimate}
\Prob\left(\sup_{n\ge 1}\Pi_{n-1}Q_{n}>t\right)\ \ge\ \Prob(Q>r)\,\Prob(\tau(t)<\infty)
\end{equation}
for any $t\ge r$. For a proof, we first note that
$$ \{\tau(t)=n\}\ \subset\ \{\Pi_{n-1}>\Pi_{k}\text{ for }k=0,...,n-2\} $$
for $t\ge r$. Using this, we obtain
\begin{align*}
\Prob(X>t)\ &\ge\ \Prob\left(\bigcup_{n\ge 1}\left\{\Pi_{n-1}Q_{n}>t,Q_{n}>r,M_{n-1}Q_{n}>r,...,M_{2}\cdot...\cdot M_{n-1}Q_{n}>r\right\}\right)\\
&\ge\ \Prob\left(\bigcup_{n\ge 1}\left\{\Pi_{n-1}>\frac{t}{r},Q_{n}>r,M_{n-1}Q_{n}>r,...,M_{2}\cdot...\cdot M_{n-1}Q_{n}>r\right\}\right)\\
&=\ \sum_{n\ge 1}\Prob\left(\tau(t)=n,\,Q_{n}>r,M_{n-1}Q_{n}>r,...,M_{2}\cdot...\cdot M_{n-1}Q_{n}>r\right)\\
&\ge\ \sum_{n\ge 1}\Prob\left(\tau(t)=n,\,Q_{n}>r,M_{n-1}>1,...,M_{2}\cdot...\cdot M_{n-1}>1\right)\\
&=\ \Prob(Q>r)\sum_{n\ge 1}\Prob\left(\tau(t)=n,\,\Pi_{n-1}>\Pi_{k}\text{ for }k=0,...,n-2\right)\\
&\ge\ \Prob(Q>r)\,\Prob\left(\sup_{n\ge 1}\Pi_{n-1}>\frac{t}{r}\right)
\end{align*}
and thus \eqref{eq:max-type SFPE first passage estimate} by virtue of \eqref{eq:tau(t) finite}. The desired result $C_{+}>0$ now follows because $\Prob(Q>r)>0$ and
$$ \lim_{t\to\infty}\left(\frac{t}{r}\right)^{\kappa}\Prob\left(\sup_{n\ge 1}\Pi_{n-1}>\frac{t}{r}\right)\ =\ \lim_{t\to\infty}e^{\kappa t}\,\Prob\left(\sup_{n\ge 0}S_{n}>t\right)\ >\ 0 $$
by invoking a well-known result from the theory of random walks, see Feller \cite[Ch.\ XII, (5.13)]{Feller:71}.\qed
\end{proof}

\section{Main results}

In all results presented hereafter, let $(X_{n})_{n\ge 0}$ be an IFS of iid Lipschitz functions $\Psi_{1},\Psi_{2},...$ on $(\X,d)$ with generic copy $\Psi$ and a not necessarily unique stationary law $\pi$. Let also $x_{0}\in\X$ be any fixed reference point. The lower and upper tail index of $\Prob_{\pi}(d(x_{0},X_{0})\in\cdot)$ (provided they exist) are denoted $\vth_{*}$ and $\vth^{*}$. Observe that, if the IFS is contractive and thus $\pi$ unique, then, by the almost sure convergence of the backward iterations $\wh{X}_{n}$ and the continuity of $d$ in both arguments,
\begin{equation*}
d(x_{0},X_{0})\ \eqdist\ D_{1:n}(X_{0})\ =\ d(x_{0},\wh{X}_{n})\ \stackrel{n\to\infty}{\longrightarrow}\ d(x_{0},\wh{X}_{\infty})\quad\Prob_{\pi}\text{-a.s.}
\end{equation*}

\begin{Theorem}\label{thm:IFS basic tail result}
Suppose there exist nondecreasing and continuous random functions $F_{n},G_{n}:I\to I$ for $n\ge 1$ such that $\R_{\ge}\subset I$ and (C1) and (C2) of Lemma \ref{lem:basic lemma} are valid for some $x_{0}\in\X$. Suppose further that the IFS generated by $(F_{n})_{n\ge 1}$ and $(G_{n})_{n\ge 1}$ are both contractive with almost sure backward limits $\wh{Y}_{\infty},\wh{Z}_{\infty}$ and unique stationary laws $\pi_{F},\pi_{G}$ having tail indices $\vth_{F}$ and $\vth_{G}$, respectively. If $X_{0}\eqdist\pi$, then
\begin{equation}\label{eq:general iteration tail inequality}
\Prob(\wh{Y}_{\infty}>t)\ \le\ \Prob(d(x_{0},X_{0})>t)\ \le\ \Prob(\wh{Z}_{\infty}>t)
\end{equation}
for all $t\in\R_{\ge}$, and a fortiori
\begin{equation}\label{eq:backward iteration inequality}
\wh{Y}_{\infty}\ \le\ d(x_{0},\wh{X}_{\infty})\ \le\ \wh{Z}_{\infty}\quad\text{a.s.}
\end{equation}
if $(X_{n})_{n\ge 0}$ is contractive with a.s. backward limit $\wh{X}_{\infty}$.
Furthermore,
\begin{equation}\label{eq:tail inequality}
\vth_{G}\ \le\ \vth_{*}\ \le\ \vth^{*}\ \le\ \vth_{F}.
\end{equation}
Finally, if $\vth_{F}=\vth_{G}=:\vth$ is the exact tail index of $\pi_{F}$ and $\pi_{G}$, then it is also the exact tail index of $\Prob_{\pi}(d(x_{0},X_{0})\in\cdot)$.
\end{Theorem}

\begin{proof}
Suppose $X_{0}\eqdist\pi$ and thus $X_{n}\eqdist\pi$ as well as $D_{1:n}(X_{0})\eqdist d(x_{0},X_{0})$ for all $n\ge 1$. By Elton's result,
$$ F_{1:n}(X_{0})\ \stackrel{n\to\infty}{\longrightarrow}\ \wh{Y}_{\infty}\ \eqdist\ \pi_{F}\quad\text{and}\quad G_{1:n}(X_{0})\ \stackrel{n\to\infty}{\longrightarrow}\ \wh{Z}_{\infty}\ \eqdist\ \pi_{G}\quad\text{a.s.} $$
and since $F_{1:n}(d(x_{0},x))\le D_{1:n}(x)\le G_{1:n}(d(x_{0},x))$ for all $x\in\X$ and $n\ge 1$ by Lemma \ref{lem:basic lemma}, we see that \eqref{eq:general iteration tail inequality} holds which in turn entails \eqref{eq:tail inequality}. In the contractive case we also infer \eqref{eq:backward iteration inequality}, for
\begin{equation*}
\wh{Y}_{\infty}\ \le\ d(x_{0},\wh{X}_{\infty})\ =\ \lim_{n\to\infty}D_{1:n}(X_{0})\ \le\ \wh{Z}_{\infty}\quad\text{a.s.}
\end{equation*}
The final assertion is trivial.\qed
\end{proof}

As the next lemma shows, very simple nondecreasing and continuous random functions $F_{n},G_{n}:\R_{\ge}\to\R_{\ge}$ of the kind discussed in Subsections \ref{subsec:perp} and \ref{subsec:Lindley} may always be provided for the given IFS $(X_{n})_{n\ge 0}$, which then leaves us with the task of giving conditions on the IFS generated by $(F_{n})_{n\ge 1}$ and $(G_{n})_{n\ge 1}$ such that the previous theorem is applicable. This is where implicit renewal theory enters.

\begin{Lemma}\label{lem:IFS crucial bounds}
For any random Lipschitz function $\Psi\in\cC_{Lip}(\X)$ and any $r>0$,
\begin{equation}\label{eq:sandwich inequality}
F(d(x_{0},x))\ \le\ d(x_{0},\Psi(x))\ \le\ G(d(x_{0},x))
\end{equation}
for all $x\in\X$, where for $t\in\R_{\ge}$
\begin{align*}
&F(t)\ :=\ \uline{Q}_{\Psi}(r)\vee\big(\uline{M}_{\Psi}(r)\,t\,\1_{(r,\infty)}(t)\big)\quad\text{and}\quad G(t)\ :=\ \ovl{Q}_{\Psi}(r)+\ovl{M}_{\Psi}(r)\,t\\
&\text{with}\quad\uline{M}_{\Psi}(r)\ :=\ \inf_{x:d(x_{0},x)>r}\frac{d(x_{0},\Psi(x))}{d(x_{0},x)},\\
&\phantom{with}\quad\ovl{M}_{\Psi}(r)\ :=\ \sup_{x:d(x_{0},x)>r}\frac{d(x_{0},\Psi(x))}{d(x_{0},x)},\\
&\phantom{with}\quad\uline{Q}_{\Psi}(r)\ :=\ \inf_{x:d(x_{0},x)\le r}d(x_{0},\Psi(x)),\\
&\text{and}\quad \ovl{Q}_{\Psi}(r)\ :=\ \sup_{x:d(x_{0},x)\le r}d(x_{0},\Psi(x)).
\end{align*}
\end{Lemma}

We note in passing that the lemma remains obviously valid when replacing the random variable $\uline{Q}_{\Psi}(r)$ with the smaller
$$ \uline{Q}_{\Psi}\ :=\ \inf_{x\in\X}d(x_{0},\Psi(x))\ =\ \inf_{r>0}\,\uline{Q}_{\Psi}(r) $$ 
in the definition of $F$.

\begin{proof}
Trivial when observing that $\uline{M}_{\Psi}(r)d(x_{0},x)\le d(x_{0},\Psi(x))\le\ovl{M}_{\Psi}(r)d(x_{0},x)$ on the set $\{x:d(x_{0},x)>r\}$.\qed
\end{proof}

For the ease of notation, we simply write $\uline{M}(r),\,\ovl{M}(r),...$ for $\uline{M}_{\Psi}(r),\,\ovl{M}_{\Psi}(r),...$ hereafter. Note also that, as $r\to\infty$,
\begin{align*}
\uline{M}(r)\ \uparrow\ &\uline{M}\ :=\ \liminf_{x:d(x_{0},x)\to\infty}\frac{d(x_{0},\Psi(x))}{d(x_{0},x)}\\
\text{and}\quad\ovl{M}(r)\ \downarrow\ &\ovl{M}\ :=\ \limsup_{x:d(x_{0},x)\to\infty}\frac{d(x_{0},\Psi(x))}{d(x_{0},x)}.
\end{align*}

\begin{Theorem}
(a)\quad Suppose for some $r>0$ the following assumptions be true:
\begin{description}[(TB-2)]
\item[(TB-1)] $\uline{M}(r),\,\ovl{M}(r)$ both satisfy (IRT-1)-(IRT-3) with
$\kappa=\alpha(r)$ and $\kappa=\beta(r)$, respectively.\vspace{.1cm}
\item[(TB-2)] $\Prob(\uline{Q}(r)>r)>0$ and $\Erw\uline{Q}(r)^{\alpha(r)}<\infty$ (or the same conditions for $\uline{Q}$).\vspace{.1cm}
\item[(TB-3)] $0<\Erw\ovl{Q}(r)^{\beta(r)}<\infty$.
\end{description}
\begin{description}[(c)]
\item[] Then $\beta(r)\le\vth_{*}\le\vth^{*}\le\alpha(r)$.\vspace{.2cm}
\item[(b)] If the previous assumptions hold for all sufficiently large $r>0$, then 
$$ \beta\le\vth_{*}\le\vth^{*}\le\alpha, $$
where $\alpha=:\lim_{r\to\infty}\alpha(r)$ and $\beta:=\lim_{r\to\infty}\beta(r)$.\vspace{.2cm}
\item[(c)] If $\alpha=\beta$ in the situation of (b), then $\pi$ has tail index $\alpha$.\vspace{.2cm}
\item[(d)] If $\alpha=\beta$ and $\Erw\ovl{M}(s)^{\alpha(s)}<\infty$ for some $s>0$, then $\uline{M}=\ovl{M}=:M$ and $M$ satisfies (IRT-1) and (IRT-2) with $\kappa=\alpha$.
\end{description}
\end{Theorem}

\begin{proof}
(a) By Lemma \ref{lem:IFS crucial bounds},
$$ \uline{Q}(r)\vee\big(\uline{M}(r)\,d(x_{0},x)\,\1_{(r,\infty)}(d(x_{0},x))\big)\ \le\ d(x_{0},\Psi(x))\ \le\ \ovl{Q}(r)+\ovl{M}(r)d(x_{0},x) $$
for all $x\in\X$. Moreover, $\Prob(\uline{Q}(r)>r)>0$ by (TB-2) and $\Prob(\ovl{Q}(r)>0)>0$.
Now it is readily seen that Prop. \ref{prop:tails of perpetuities} (and the following remark) and Prop. \ref{prop:tails of max-type SFPE} can be used to infer that the stationary laws $\pi_{F}$ and $\pi_{G}$ of the IFS pertaining to $F(t):=\uline{Q}(r)\vee\big(\uline{M}(r)\,t\,\1_{(r,\infty)}(t)\big)$ and $G(t):=\ovl{Q}(r)+\ovl{M}(r)\,t$ have exact tail indices $\alpha(r)$ and $\beta(r)$, respectively. Consequently, the assertion follows with the help of Thm. \ref{thm:IFS basic tail result}.

\vspace{.2cm}
(b) Here it suffices to note that $\alpha(r)$ decreases and $\beta(r)$ increases in $r$.

\vspace{.2cm}
(c) is trivial.

\vspace{.2cm}
(d) If $\alpha=\beta$ and $\Erw\ovl{M}(s)^{\alpha(s)}<\infty$ for some $s>0$, then it follows from
\begin{align*}
&\sup_{r\ge s}\Erw\uline{M}(r)^{\alpha(r)}\1_{\{\uline{M}(r)>t\}}\ \le\ \Erw\ovl{M}(s)^{\alpha(s)}\1_{\{\ovl{M}(s)>t\}}\ \stackrel{t\to\infty}{\longrightarrow}\ 0\\
\text{and}\quad&\sup_{r\ge s}\Erw\ovl{M}(r)^{\beta(r)}\1_{\{\ovl{M}(r)>t\}}\ \le\ \Erw\ovl{M}(s)^{\alpha}\1_{\{\ovl{M}(s)>t\}}\ \stackrel{t\to\infty}{\longrightarrow}\ 0
\end{align*}
that $\{\uline{M}(r)^{\alpha(r)}:r\ge s\}$ and $\{\ovl{M}(r)^{\beta(r)}:r\ge s\}$ are uniformly integrable which in combination with $\lim_{r\to\infty}\uline{M}(r)^{\alpha(r)}=\uline{M}^{\alpha}$ and $\lim_{r\to\infty}\ovl{M}(r)^{\beta(r)}=\ovl{M}^{\alpha}$ a.s. entails $\Erw\uline{M}^{\alpha}=\Erw\ovl{M}^{\alpha}=1$ and thus also $\uline{M}=\ovl{M}$ a.s. Finally, $M$ satisfies (IRT-2), i.e. $\Erw M^{\alpha}\log^{+}M<\infty$, because either $\alpha(s)=\alpha$ and so $M=\ovl{M}(s)$ a.s., or $\alpha(s)>\alpha$ and $\Erw M^{\alpha(s)}\le\Erw\ovl{M}^{\alpha(s)}<\infty$.\qed
\end{proof}

\begin{Theorem}\label{thm:IFS exact TI}
Suppose there exist $r\ge 0$ and nonnegative random variables $M,\uline{R},\ovl{R}$ such that \eqref{eq:sandwich inequality} holds for all $x\in\X$ with
$$ F(t) :=\ \uline{R}\vee\big(M\,t\,\1_{(r,\infty)}(t)\big)\quad\text{and}\quad G(t)\ :=\ \ovl{R}+M\,t. $$
Then $\pi$ has exact tail index $\kappa$ provided that $\Prob(\uline{R}>r)>0,\,\Prob(\ovl{R}>0)>0$, $\Erw\ovl{R}^{\kappa}<\infty$ and $M$ satisfies (IRT-1)-(IRT-3).
\end{Theorem}

\begin{proof}
By another appeal to Prop. \ref{prop:tails of perpetuities} and Prop. \ref{prop:tails of max-type SFPE}, we infer that the stationary laws $\pi_{F}$ and $\pi_{G}$ of the IFS pertaining to $F$ and $G$, respectively, both have exact tail index $\kappa$ which, by Thm. \ref{thm:IFS basic tail result}, is therefore also the tail index of $\pi$. Exactness finally follows because, by \eqref{eq:general iteration tail inequality},
\begin{align*}
0\ <\ \lim_{t\to\infty}t^{\kappa}\,\Prob(\wh{Y}_{\infty}>t)\ &\le\ \liminf_{t\to\infty}t^{\kappa}\,\Prob(d(x_{0},X_{0})>t)\\
&\le\ \limsup_{t\to\infty}t^{\kappa}\,\Prob(d(x_{0},X_{0})>t)\ \le\ \lim_{t\to\infty}t^{\kappa}\,\Prob(\wh{Z}_{\infty}>t)\ <\ \infty
\end{align*}
where $X_{0}\eqdist\pi,\,\wh{Y}_{\infty}\eqdist\pi_{F}$ and $\wh{Z}_{\infty}\eqdist\pi_{G}$.\qed
\end{proof}

\section{Examples}\label{sec:examples}

\subsection{The AR(1)-model with ARCH(1) errors}\label{subsec:AR-ARCH}

The following IFS, which belongs to a larger class of nonlinear time series models introduced by Engle \cite{Engle:82} and Weiss \cite{Weiss:84}, has received attention due to its relevance in Mathematical Finance where it is known as a relatively simple model that captures temporal variation of volatility in financial data sets (conditional heteroscedasticity). Known as the \emph{AR(1)-model with ARCH(1) errors}, it is defined by the recursion
\begin{align*}
X_{n}\ =\ \alpha X_{n-1}+\left(\beta+\lambda X_{n-1}^{2}\right)^{1/2}\eps_{n},\quad n\ge 1,
\end{align*}
with $(\alpha,\beta,\lambda)\in\R\times\R_{>}^{2}$ being a parameter. The $\eps_{n}$, called innovations, are assumed to be independent of $X_{0}$ and further iid with a nontrivial \emph{symmetric} distribution. Regarding existence and tail behavior of the stationary distribution, a detailed study and relatively explicit results for the case $\alpha=0$ and standard normal $\eps_{n}$ (ARCH(1)-model with Gaussian noise) may be found in the monograph by Embrechts, Kl\"uppelberg and Mikosch \cite[Section 8.4, especially Thm. 8.4.9]{EmKlMi:97}. The more difficult general case was treated by Borkovec and Kl\"uppelberg \cite{BorkovecKl:01} who particularly provided, by a rather long and technical Tauberian-type argument [see their Section 4], the tail index of the stationary law under some extra conditions on the law of the $\eps_{n}$. Theorem \ref{thm:tail AR(1)-ARCH} below not only improves their result by showing that the tail index is actually exact, but is also obtained by much simpler means using our sandwich technique.

Let $\eps$ denote a generic copy of the $\eps_{n}$. If $(X_{n})_{n\ge 0}$ has a unique stationary law $\pi$, then any random variable $X$ with law $\pi$ and independent of $\eps$ satisfies the SFPE
\begin{equation}\label{eq:SFPE AR(1)-ARCH}
X\ \eqdist\ \Phi(X)\ :=\ \alpha X+\big(\beta+\lambda X^{2}\big)^{1/2}\eps
\end{equation}
and is symmetric, for $-X$ also solves \eqref{eq:SFPE AR(1)-ARCH}. Moreover, it then further follows that
\begin{equation*}
X\ \eqdist\ (-\alpha)(-X)+\big(\beta+\lambda (-X)^{2}\big)^{1/2}\eps\ \eqdist\ -\alpha X+\big(\beta+\lambda X^{2}\big)^{1/2}\eps,
\end{equation*}
whence it is no loss of generality to assume $\alpha\ge 0$. The symmetry of $X$ also allows us to study the tail of $W:=X^{2}$, for
$$ \Prob(X>t)\ =\ \frac{1}{2}\Prob(|X|>t)\ =\ \frac{1}{2}\,\Prob(W>t^{2}) $$
for all $t\ge 0$.

\vspace{.1cm}
It is not difficult to verify that $L(\Phi)=\alpha+\lambda^{1/2}|\eps|$ and then that the IFS $(X_{n})_{n\ge 0}$ is mean contractive and satisfies the jump-size condition \eqref{cond:jump-size} (with $x_{0}=0$) if
\begin{equation}\label{eq:contraction AR(1)-ARCH}
\Erw\log(\alpha+\lambda^{1/2}|\eps|)<0,
\end{equation}
in particular $\alpha<1$.
Assuming beyond symmetry that the law of $\eps$ has a continuous Lebesgue density and finite second moment and that its support is the whole real line, Borkovec and Kl\"uppelberg \cite[Thm.~1]{BorkovecKl:01} could actually show, by drawing on the theory of Harris recurrence, that $(X_{n})_{n\ge 0}$ has a unique stationary distribution already under the weaker condition
\begin{equation}\label{eq2:contraction AR(1)-ARCH}
\Erw\log|\alpha+\lambda^{1/2}\eps|<0.
\end{equation}
Here we contend ourselves with condition \eqref{eq:contraction AR(1)-ARCH}, but do not impose restictions of the afore-mentioned kind on the law of $\eps$.

\vspace{.1cm}
Note that $|X|$ is independent of the random variable $\sign(X)$ which in turn takes values $\pm 1$ with probability $1/2$ each. Hence $\eta:=\sign(X)\eps$ is a copy of $\eps$ independent of $|X|$ and thus of $W=|X|^{2}$. This in combination with \eqref{eq:SFPE AR(1)-ARCH} entails that 
\begin{align}
W\ &\eqdist\ (\alpha^{2}+\lambda\eps^{2})W+2\alpha\eps X(\beta+\lambda W)^{1/2}+\beta\eps^{2}\nonumber\\
&\eqdist\ (\alpha+\lambda^{1/2}\eta)^{2}\,W+2\alpha\eta W^{1/2}\Big((\beta+\lambda W)^{1/2}-(\lambda W)^{1/2}\Big)+\beta\eta^{2},\nonumber
\intertext{thus}
W\ &\eqdist\ \Psi(W)\ :=\ (\alpha+\lambda^{1/2}\eta)^{2}\,W+\frac{2\alpha\beta\eta W^{1/2}}{(\beta+\lambda W)^{1/2}+(\lambda W)^{1/2}}+\beta\eta^{2}.\label{eq:SFPE W}
\end{align}
The random Lipschitz function $\Psi$ is easily seen to satisfy the sandwich inequality
\begin{equation}\label{eq:sandwich AR(1)-ARCH}
\uline{R}+M\,t=:F(t)\ \le\ \Psi(t)\ \le\ G(t):=\ovl{R}+M\,t,\quad t\in\R_{\ge},
\end{equation}
where
\begin{align*}
&\uline{R}:=\beta\left(\eta^{2}-\alpha\lambda^{-1/2}\eta^{-}\right),\quad \ovl{R}:=\beta\left(\eta^{2}+\alpha\lambda^{-1/2}|\eta|\right)\\
\text{and}\quad &M:=(\alpha+\lambda^{1/2}\eta)^{2}.
\end{align*}
An application of Thm. \ref{thm:IFS exact TI} now leads to the following result.

\begin{Theorem}\label{thm:tail AR(1)-ARCH}
Given any $(\alpha,\beta,\lambda)\in\R\times\R_{>}^{2}$, assume \eqref{eq:contraction AR(1)-ARCH} and that $M$ as above satisfies (IRT-1)-(IRT-3) for some $\kappa>0$. Then the solution $\pi$ to the SFPE \eqref{eq:SFPE AR(1)-ARCH} is unique and
\begin{equation*}
\lim_{t\to\infty}t^{2\kappa}\,\Prob(X>t)\ =\ \frac{\Erw\big(\Psi(X^{2})^{\kappa}-(MX^{2})^{\kappa}\big)}{2\kappa\mu_{\kappa}}\ >\ 0,
\end{equation*}
where $\mu_{\kappa}:=\Erw M^{\kappa}\log M$ and $X\eqdist\pi$.
\end{Theorem}

\begin{proof}
W.l.o.g. assume $\alpha\ge 0$. If \eqref{eq:contraction AR(1)-ARCH} holds, then the jump-size condition 
$$ \Erw\log^{+}|\Phi(0)|=\Erw\log^{+}\beta^{1/2}|\eps|<\infty $$
is easily seen to be valid as well. Consequently, by Elton's theorem, the IFS $(X_{n})_{n\ge 0}$ is mean contractive with unique stationary distribution $\pi$ and the backward iterations $\wh{X}_{n}$ converge a.s. Put $W_{n}:=\wh{X}_{n}^{2}$ with $\wh{X}_{0}\eqdist\pi$ and $W=\lim_{n\to\infty}W_{n}$.

\vspace{.1cm}
Next, if $M$ satisfies (IRT-1)-(IRT-3), then $\Erw|\eta|^{2\kappa}<\infty$ and $\Erw|\uline{R}|^{\kappa}\le\Erw\ovl{R}^{\kappa}<\infty$. Moreover, the IFS generated by $F$ and $G$ are mean contractive. Denote by $\pi_{F},\,\pi_{G}$ their stationary distributions, respectively, and by $Y,Z$ their a.s. backward iteration limits. Then Thm.~\ref{thm:IFS basic tail result} ensures that $Y\le W\le Z$ a.s.

\vspace{.1cm}
Since $\Erw|\uline{R}|^{\kappa}<\infty$ and $\Erw\ovl{R}^{\kappa}<\infty$,
we also have $\Erw|Y|^{p}<\infty$ and $\Erw W^{p}\le\Erw Z^{p}<\infty$ for any $p\in (0,\kappa)$, see Prop.~\ref{prop:tails of perpetuities}. Therefore
\begin{align}\label{eq:moment ARCH}
\Erw\big(\Psi(W)^{\kappa}-(MW)^{\kappa}\big)\ \le\ \Erw\big((MW+\ovl{R})^{\kappa}-(MW)^{\kappa}\big)\ <\ \infty.
\end{align}
for the last expectation is bounded by $\Erw\ovl{R}^{\kappa}$ if $\kappa\in (0,1]$ (subadditivity), and by
$\Erw\ovl{R}^{\kappa}$ plus a constant times
\begin{align*}
\Erw W^{\kappa-1}\,\Erw\ovl{R}\ +\ \Erw W\,\Erw\ovl{R}^{\kappa-1}\ <\ \infty
\end{align*}
if $\kappa>1$. For the last estimate, we have used that, for $x,y\ge 0$,
\begin{align*}
(x+y)^{\kappa}-x^{\kappa}\ &\le\ y^{\kappa}+\kappa 2^{\kappa-1}(x^{\kappa-1}y+xy^{\kappa-1}),
\end{align*}
see \cite[p. 282]{Gut:74} and also \cite[(9.27)]{Goldie:91} for a similar estimate.

By invoking once again Thm.~\ref{thm:IFS basic tail result} in combination with Prop.~\ref{prop:IRT}, we now infer
\begin{align*}
\lim_{t\to\infty}2t^{2\kappa}\,\Prob(X>t)\ &=\ 
\lim_{t\to\infty}t^{\kappa}\,\Prob(W>t)\ =\ \frac{\Erw\big(\Psi(W)^{\kappa}-(MW)^{\kappa}\big)}{\kappa\mu_{\kappa}}\\
&\ge\ \lim_{t\to\infty}t^{2\kappa}\,\Prob(Y>t)\ =\ \frac{\Erw\big((MY+\uline{R})^{\kappa}-(MY)^{\kappa}\big)}{\kappa\mu_{\kappa}}
\end{align*}
so that we must finally verify that the last expectation is positive.

\vspace{.1cm}
To this end note that $\Prob(M=1)<1$ and $\Erw M^{\kappa}=1$ imply
\begin{align*}
&0\ <\ \Prob(M<1)\ \le\ \Prob(-(1+\alpha)\lambda^{-1/2}<\eta<(1-\alpha)\lambda^{-1/2})\\
\text{and}\quad&0\ <\ \Prob(M>1)\ \le\ \Prob(\eta<-(1+\alpha)\lambda^{-1/2}\text{ or }\eta>(1-\alpha)\lambda^{-1/2})
\end{align*}
and therefore (using the symmetry of $\eta$)
\begin{align*}
&\Prob(M\le 1,\,\uline{R}>0)\ \ge\ \Prob(0<\eta<(1-\alpha)\lambda^{-1/2})\ >\ 0
\intertext{as well as}
&\Prob(M\ge 1,\,\uline{R}>0)\ \ge\ \Prob(\eta>(1-\alpha)\lambda^{-1/2})\ >\ 0.
\end{align*}
Hence condition \eqref{eq2:no invariante halfline RDE} holds for the pair $(M,\uline{R})$ and we arrive at the desired conclusion by Prop.~\ref{prop:Guivarch perpetuity}.\qed
\end{proof}

\subsection{Random logistic transforms}\label{subsec:logistic}

A \emph{random logistic transform} is given by the Lipschitz function
$$ \Phi(x)\ :=\ \xi^{-1}\,x(1-x),\quad x\in [0,1], $$
where $\xi$ denotes a random variable taking values in $[1/4,\infty)$, where the last restriction is necessary to ensure $\Phi([0,1])\subset [0,1]$. An IFS generated by iid copies of $\Phi$, that is
\begin{equation*}
X_{n}\ =\ \Phi_{n}(X_{n-1})\ =\ \xi_{n}^{-1}X_{n-1}(1-X_{n-1}),\quad n\ge 1,
\end{equation*} 
has been studied in a series of papers of which we mention those by Athreya and Dai \cite{AthreyaDai:00,AthreyaDai:02}, Dai \cite{Dai:00}, Athreya and Schuh \cite{AthreyaSchuh:03}, Steinsaltz \cite{Steinsaltz:01} and the survey by Athreya and Bhattacharya \cite{AthreyaBhatt:01}. The contractive case, which occurs if $\Erw\log\xi>0$, is rather uninteresting here because it results in the trivial stationary distribution $\delta_{0}$. As shown in \cite[Thm. 5]{AthreyaDai:00}, the same along with $X_{n}\stackrel{\Prob}{\to} 0$ holds true when $\Erw\log\xi=0$ (called critical case), where $\stackrel{\Prob}{\to}$ means convergence in probability. In fact, $\delta_{0}$ is always stationary because $\Phi(0)=0$ for any realization of $\xi$. On the other hand, if
\begin{equation}\label{eq:extra moment at 4}
-\infty<\Erw\log\xi<0\quad\text{and}\quad\Erw|\log(4\xi-1)|<\infty,
\end{equation}
there exists also a stationary distribution $\pi$ on the open interval $(0,1)$ which is unique if $(X_{n})_{n\ge 0}$ is Harris irreducible on $(0,1)$, see \cite[Thms. 2 and 6]{AthreyaDai:00}. It is then natural to ask about the behavior of $\pi$ at $0$, more precisely, of $\pi((0,x))$ as $x\downarrow 0$. The following considerations will show how this may be accomplished within our framework under additional conditions on $\xi$.

\vspace{.1cm}
After conjugation with $x\mapsto x^{-1}$, the IFS $(X_{n})_{n\ge 0}$ turns into the IFS $(W_{n})_{n\ge 0}$, defined by the recursion
\begin{equation*}
W_{n}\ =\ \Psi_{n}(W_{n-1})\ :=\ \frac{1}{\Phi_{n}(1/W_{n-1})}\ =\ \xi_{n}\,\left(W_{n-1}+1+\frac{1}{W_{n-1}-1}\right)
\end{equation*}
for $n\ge 1$ and with state space $\X=(1,\infty)$. It has stationary distribution $\wh{\pi}$, given by
\begin{equation}\label{eq:def pihat}
\wh{\pi}((x,\infty))\ :=\ \pi((0,1/x)),\quad x>1.
\end{equation}
In order to study its tail behavior, we first consider the simpler case when $4\xi$ stays bounded away from 1, the result being summarized in the next theorem.

\begin{Theorem}\label{thm:tail logistic transform}
Suppose that $\xi$ satisfies \eqref{eq:extra moment at 4} and (IRT-1)-(IRT-3) for some $\kappa>0$. Then the following assertions hold true for any stationary distribution $\pi$ of $(X_{n})_{n\ge 0}$ on $(0,1)$:
\begin{description}[(c)]
\item[(a)] $\wh{\pi}$ has upper tail index $\kappa$. In fact,
\begin{equation}\label{eq:upper tail index pihat}
\liminf_{x\to\infty}x^{\kappa}\,\wh{\pi}((x,\infty))\ =\ \liminf_{x\downarrow 0}\,x^{-\kappa}\pi((0,x))\ \in\ \R_{>}\cup\{\infty\}.
\end{equation}
\item[(b)] If $4\xi\ge a\in (1,4)$ a.s., then $\pi((0,1/a])=1$ and
\begin{equation}\label{eq:exact tail index pihat}
\lim_{x\downarrow 0}\,x^{-\kappa}\pi((0,x))\ =\ \frac{\Erw\big(\Phi(X)^{-\kappa}-(X/\xi)^{-\kappa}\big)}{\kappa\mu_{\kappa}}\ \in\ \R_{>},
\end{equation}
where $X$ is independent of $\Phi$ with $X\eqdist\pi$ and $\mu_{\kappa}=\Erw\xi^{\kappa}\log\xi$. In particular, $\wh{\pi}$ has exact tail index $\kappa$.
\item[(c)] If $a\ge 1+2^{-1/2}$, or if $a<1+2^{-1/2}$ and $\Erw\log\xi+\log|1-(1-a)^{-2}|<0$ in (b), then $\pi$ is unique.
\end{description}
\end{Theorem}

We note in passing that $a<4$ has been asumed, for otherwise $4\xi\ge a\ge 4$ a.s. would entail $\Erw\log\xi\ge 0$ and thus violation of \eqref{eq:extra moment at 4}.

\begin{proof}
(a) Obviously, $\Psi(x)\ge\xi(x+1)=:F(x)$, and the IFS generated by $F$ is mean contractive with unique stationary law $\pi_{F}$ satisfying
\begin{equation}\label{eq:pi_F tail}
\lim_{t\to\infty}t^{\kappa}\,\Prob(Y>t)\ >\ 0\quad\text{if }Y\eqdist\pi_{F}
\end{equation}
by Prop.~\ref{prop:tails of perpetuities} and the subsequent remark. Now use $F_{1:n}(x)\le\Psi_{1:n}(x)$ for all $x>1$ and $n\ge 1$ (Lemma \ref{lem:basic lemma}), $\Psi_{1:n}(W)\eqdist\wh{\pi}$ and $F_{1:n}(W)\to\wh{Y}_{\infty}\eqdist Y$ to infer
\begin{align*}
\wh{\pi}((t,\infty))\ =\ \Prob(W>t)\ =\ \Prob(\Psi_{1:n}(W)>t)\ \ge\ \Prob(F_{1:n}(W)>t)\ \stackrel{n\to\infty}{\longrightarrow}\ \Prob(Y>t)
\end{align*}
for all $t>1$ and thereupon \eqref{eq:upper tail index pihat}.

\vspace{.1cm}
(b) If $4\xi\ge a\in (1,4)$, then $\Psi(x)\ge\Psi(2)=4\xi\ge a$ a.s. for all $x>1$ and thus $[a,\infty)$ is an absorbing set for the IFS $(W_{n})_{n\ge 0}$ generated by $\Psi$. In particular, $\wh{\pi}([a,\infty))=\pi((0,1/a])=1$ for any stationary law $\pi$. Moreover, the sandwich inequality $F(x)\le\Psi(x)\le G(x)$ holds true a.s. on $[a,\infty)$, where $F$ is as in (a) and 
$$ G(x)\ :=\ \xi\left(x+1+\frac{1}{a-1}\right). $$
Since the IFS generated by $G$ is clearly also mean contractive with unique stationary law $\pi_{G}$ and satisfies
$$ \lim_{t\to\infty}t^{\kappa}\,\Prob(Z>t)\ >\ 0 $$
when $Z\eqdist\pi_{G}$, we infer from Thm.~\ref{thm:IFS basic tail result} that $\wh{\pi}$ has exact tail index $\kappa$. 

Now let $X\eqdist\pi$ be independent of $\Phi$, thus $X\eqdist\Phi(X)$, $W:=1/X\eqdist\wh{\pi}$ and $W\eqdist\Psi(W)$. Observe that
$$ \Phi(X)^{-\kappa}-(X/\xi)^{-\kappa}\ =\ \Psi(W)^{\kappa}-(\xi W)^{\kappa}\ >\ 0. $$
By similar arguments as in the previous example (see around \eqref{eq:moment ARCH}), we find that
\begin{align*}
0\ <\ \Erw\big(\Psi(W)^{\kappa}-(\xi W)^{\kappa}\big)\ \le\ \Erw\big(G(W)^{\kappa}-(\xi W)^{\kappa}\big)\ <\ \infty
\end{align*}
because $\Erw\xi^{\kappa}<\infty$ and $\Erw W^{p}\le\Erw Z^{p}<\infty$ for any $p\in (0,\kappa)$. Hence \eqref{eq:exact tail index pihat} follows by an appeal to Prop.~\ref{prop:IRT}.

\vspace{.1cm}
(c) For general $a\in (1,4)$, we note that the Lipschitz constant $L(\Psi)$ on $\X=[a,\infty)$ is given by the maximum of $|\Psi'(a)|=\xi|1-(a-1)^{-2}|$ and $\Psi'(\infty)=\xi$ by the convexity of $\Psi$. It equals the first of these values iff $a\in (1,1+2^{-1/2}]$. Therefore, mean contractivity of $(W_{n})_{n\ge 0}$ on $[a,\infty)$ holds iff $a\ge 1+2^{-1/2}$, or $a<1+2^{-1/2}$ and $\Erw\log\xi+\log|1-(1-a)^{-2}|<0$.\qed
\end{proof}

Turning to the general case, we will show that \eqref{eq:exact tail index pihat} remains valid under an extra moment condition which controls the behavior of $\xi$ at its lower bound $1/4$. For the proof, furnished by two subsequent lemmata, the basic and rather standard idea is to first consider an embedded IFS.

\begin{Theorem}\label{thm2:tail logistic transform}
Suppose that $\xi$ satisfies \eqref{eq:extra moment at 4}, (IRT-1)-(IRT-3) for some $\kappa>0$ and
\begin{equation}\label{eq:extra moment at 4 kappa}
\Erw(4\xi-1)^{-\kappa}<\infty.
\end{equation}
Then any stationary distribution $\pi$ of $(X_{n})_{n\ge 0}$ on $(0,1)$ satisfies \eqref{eq:exact tail index pihat}. In particular, $\wh{\pi}$ has exact tail index $\kappa$.
\end{Theorem}

Let $\sigma_{0}:=0$ and $\sigma_{n}:=\inf\{k>\sigma_{n-1}:\xi_{k}\ge 1/2\}$ for $n\ge 1$.
Let further $(W_{n}^{*})_{n\ge 0}$ be the IFS on $[2,\infty)$ generated by $\Psi_{\sigma:1}$, where $\sigma:=\sigma_{1}$. Thus $W_{n}^{*}=\Psi_{\sigma_{n}:1}(W_{0}^{*})$. The following lemma shows that it has a stationary law $\wh{\pi}^{*}$, say, with lower tail index at least as big as $\wh{\pi}$.

\begin{Lemma}\label{lem:RLT geom sampling}
For all $t>1$, $\wh{\pi}^{*}((t,\infty))\ge\wh{\pi}((2t,\infty))$.
\end{Lemma}

\begin{proof}
As one readily check, the Markov chain $(W_{n},\xi_{n})_{n\ge 0}$ has stationary law
$$ \Prob_{\pi}((\Psi_{1}(W_{0}),\xi_{1})\in\cdot), $$
which in turn implies that
\begin{equation}\label{eq:def pihatstar}
\wh{\pi}^{*}\ :=\ \frac{\Prob_{\pi}((\Psi_{1}(W_{0}),\xi_{1})\in\cdot\,,\,\xi_{1}\ge 1/2)}{\Prob_{\pi}(\xi_{1}\ge 1/2)}.
\end{equation}
is a stationary distribution of the associated hit chain $(W_{n}^{*},\sigma_{n})_{n\ge 0}$, the hitting set being $[2,\infty)\times\{1/2\}$. Now observe that
\begin{align*}
\wh{\pi}^{*}((t,\infty))\ &=\ \frac{\Prob_{\pi}(\Psi_{1}(W_{0})>t\,,\,\xi_{1}\ge 1/2)}{\Prob_{\pi}(\xi_{1}\ge 1/2)}\\
&=\ \frac{\Prob_{\pi}(\xi_{1}(W_{0}+1+(W_{0}-1)^{-1})>t\,,\,\xi_{1}\ge 1/2)}{\Prob_{\pi}(\xi_{1}\ge 1/2)}\\
&\ge\ \frac{\Prob_{\pi}(W_{0}>2t\,,\,\xi_{1}\ge 1/2)}{\Prob_{\pi}(\xi_{1}\ge 1/2)}\\
&=\ \Prob_{\pi}(W_{0}>2t)\\ 
&=\ \wh{\pi}((2t,\infty))
\end{align*}
for all $t>1$, where in the penultimate line we have used that $W_{0}$ is independent of $\xi_{1}$ with law $\wh{\pi}$ under $\Prob_{\pi}$.\qed
\end{proof}

\begin{Lemma}\label{lem:tail of Wsigma}
Under the assumptions of Thm.~\ref{thm2:tail logistic transform}, the law $\wh{\pi}^{*}$ defined by \eqref{eq:def pihatstar} satisfies
\begin{equation}\label{eq:tail of Wsigma}
C_{+}^{*}\ :=\ \lim_{t\to\infty}x^{\kappa}\,\wh{\pi}^{*}((x,\infty))\ \in\ \R_{>}
\end{equation}
and has thus exact tail index $\kappa$.
\end{Lemma}

\begin{proof}
Put $\Pi_{n}:=\xi_{1}\cdot...\cdot\xi_{n}$ for $n\ge 1$. Then, for any $x\ge 2$,
\begin{align*}
\Psi_{\sigma:1}(x)\ &=\ \xi_{\sigma}\left(\Psi_{\sigma-1:1}(x)+1+\frac{1}{\Psi_{\sigma-1:1}(x)-1}\right)\\
&\le\ \xi_{\sigma}\Psi_{\sigma-1:1}(x)\ +\ \xi_{\sigma}\ +\ \frac{\xi_{\sigma}}{4\xi_{\sigma-1}-1}\\
&\vdots\\
&\le\ \Pi_{\sigma}x\ +\ \sum_{n=1}^{\sigma}\xi_{n}\ +\ \sum_{n=2}^{\sigma}\frac{\xi_{n}}{4\xi_{n-1}-1}\ +\ \frac{\xi_{1}}{x-1}\\
&\le\ \Pi_{\sigma}x\ +\ \sum_{n=1}^{\sigma}\xi_{n}\ +\ \sum_{n=2}^{\sigma}\frac{\xi_{n}}{4\xi_{n-1}-1}\ +\ \xi_{1}\ =:\ G(x),
\end{align*}
and in a similar manner we find that
\begin{align*}
\Psi_{\sigma:1}(x)\ &\ge\ \xi_{\sigma}\left(\Psi_{\sigma-1:1}(x)+1\right)\ \ge\ \Pi_{\sigma}x\ +\ \sum_{n=1}^{\sigma}\xi_{n}\ \ge\ \Pi_{\sigma}x\ \vee\ \sum_{n=1}^{\sigma}\xi_{n}\ =:\ F(x).
\end{align*}
We thus see that $\Psi_{\sigma:1}$ is bounded from below and above by random affine functions, namely
$$ F(x)\ =\ \Pi_{\sigma}x\vee Q_{1}\quad\text{and}\quad G(x)\ =\ \Pi_{\sigma}x+Q_{1}+Q_{2} $$
where
$$ Q_{1}\ :=\ \sum_{n=1}^{\sigma}\xi_{n}\quad\text{and}\quad Q_{2}\ :=\ \sum_{n=1}^{\sigma}\frac{\xi_{n}}{4\xi_{n-1}-1}\ +\ \xi_{1}, $$
which are both positive random variables.
Therefore, it follows from Thm.~\ref{thm:IFS exact TI} that $\wh{\pi}^{*}$ has exact tail index $\kappa$ if we still verify that
\begin{description}[(2)]
\item[(1)] $\Pi_{\sigma}$ satisfies (IRT-1)-(IRT-3) for $\kappa>0$ as given, in particular $\Erw\log\Pi_{\sigma}\in\R_{<}$.
\item[(2)] $\Erw Q_{1}^{\kappa}<\infty$ and $\Erw Q_{2}^{\kappa}<\infty$.
\end{description}
It is then also easily seen that
$$ 0\ <\ \Erw\big(\Psi_{\sigma:1}(W^{*})^{\kappa}-(\Pi_{\sigma}W^{*})^{\kappa}\big)\ \le\ \Erw\big(G(W^{*})^{\kappa}-(\Pi_{\sigma}W^{*})^{\kappa}\big)\ <\ \infty $$
where $W^{*}$ has law $\wh{\pi}^{*}$ and is independent of all other occuring random variables. Hence \eqref{eq:tail of Wsigma} follows by an appeal to Prop.~\ref{prop:IRT}.

\vspace{.1cm}
We proceed to the proof of (1). First, $\Erw\log\Pi_{\sigma}\in\R_{<}$ follows by \eqref{eq:extra moment at 4} and Wald's identity, viz.
$$ \Erw\log\Pi_{\sigma}\ =\ \Erw\sum_{n=1}^{\sigma}\log\xi_{n}\ =\ \Erw\log\xi\,\Erw\sigma\ \in\ \R_{<}. $$
Using that $\sigma$ has a geometric distribution and putting $p:=\Erw\xi^{\kappa}\1_{\{\xi\ge 1/2\}}$, thus $q:=\Erw\xi^{\kappa}\1_{\{\xi<1/2\}}=1-p$, we further obtain
\begin{align*}
\Erw\Pi_{\sigma}^{\kappa}\ =\ \sum_{n\ge 1}\Erw\Pi_{n}^{\kappa}\1_{\{\sigma=n\}}\ =\ p\sum_{n\ge 1}q^{n}\ =\ 1
\end{align*}
as well as (noting that $0<p\vee q<1$)
\begin{align*}
\Erw\Pi_{\sigma}^{\kappa}\log\Pi_{\sigma}\ &=\ \sum_{n\ge 1}\sum_{k=1}^{n}\Erw\Pi_{n}^{\kappa}\log\xi_{k}\1_{\{\sigma=n\}}\ \le\ \sum_{n\ge 1}n(p\vee q)^{n-1}\Erw\xi^{\kappa}\log\xi\ <\ \infty.
\end{align*}
Finally, the lattice-type of $\log\Pi_{\sigma}=\sum_{n=1}^{\sigma}\log\xi_{n}$ given $\log\Pi_{\sigma}>0$ is easily seen to be the same as the lattice-type of $\log\xi$ given $\log\xi>0$.
This completes the proof of (1).

\vspace{.1cm}
Assertion (2) in the case $0<\kappa\le 1$ is easily obtained by a subadditivity argument in combination with $\Erw\xi^{\kappa}<\infty$ (by (IRT-1)) and
\begin{equation}\label{eq2:extra moment at 4 kappa}
\Erw\left(\frac{\xi_{n}}{4\xi_{n-1}-1}\right)^{\kappa}\ =\ \Erw\xi^{\kappa}\,\Erw(4\xi-1)^{-\kappa}\ =\ \Erw(4\xi-1)^{-\kappa}\ <\ \infty
\end{equation}
which is guaranteed by \eqref{eq:extra moment at 4 kappa}. So let $\kappa>1$ in which case $\Erw\xi<\infty$. Then $\Erw Q_{1}^{\kappa}<\infty$ follows directly from a standard result for stopped random walks, see Thm.~5.1 in Gut's monography \cite{Gut:09}, when decomposing $Q_{1}$ in the form
$$ Q_{1}\ =\ \sum_{n=1}^{\sigma}(\xi_{n}-\Erw\xi)\ +\ \sigma\,\Erw\xi. $$
But by another use of \eqref{eq2:extra moment at 4 kappa}, a similar result applies to the stopped sum $Q_{2}$ of the 1-dependent, almost stationary sequence
$$ \xi_{1},\ \frac{\xi_{2}}{4\xi_{1}-1},\ \frac{\xi_{3}}{4\xi_{2}-1},\ \ldots, $$
see Janson \cite[Thm. 1.3]{Janson:83}.\qed
\end{proof}

\begin{proof}[of Thm.~\ref{thm2:tail logistic transform}]
By using the two previous lemmata, we infer that
\begin{align*}
\limsup_{x\to\infty}x^{\kappa}\,\wh{\pi}((x,\infty))\ \le\ \limsup_{x\to\infty}x^{\kappa}\,\wh{\pi}^{*}((x/2,\infty))\ \in\ \R_{>}
\end{align*}
which combined with Thm.~\ref{thm:tail logistic transform}(a) proves that $\wh{\pi}$ has exact tail index $\kappa$. In particular, if $W\eqdist\wh{\pi}$, then $\Erw W^{p}<\infty$ for any $p\in (0,\kappa)$. Since $\Psi(x)\ge 4\xi$ for all $x>1$ further implies
\begin{align*}
\Prob\left(\frac{1}{W-1}>x\right)\ =\ \Prob\left(\frac{1}{\Psi(W)}>x\right)\ \le\ \Prob\left(\frac{1}{4\xi-1}>x\right)
\end{align*}
for all $x>1$ and thus $\Erw(W-1)^{-\kappa}\le\Erw(4\xi-1)^{-\kappa}<\infty$, it is finally not difficult to conclude that
\begin{align*}
\Erw\Big(\Phi(X)^{-\kappa}-(X/\xi)^{-\kappa}\Big)\ =\ \Erw\Big(\Psi(W)^{\kappa}-(\xi W)^{\kappa}\Big)\ <\ \infty
\end{align*}
and then \eqref{eq:exact tail index pihat} by an appeal to Prop.~\ref{prop:IRT}.\qed
\end{proof}

\subsection{The stochastic Ricker model}\label{subsec:Ricker}

Our next example, taken from the theory of population dynamics, is very similar in flavor to the previous one.

According to Hassell \cite{Hassell:75}, a good model for population dynamics in a limited environment should bear the following features:
\begin{itemize}
\item a potential of exponential increase when the population size is small;
\item a density-dependent feedback that progressively reduces the actual rate of increase.
\end{itemize}
A deterministic model that meets these requirements was introduced by Ricker \cite{Ricker:54} and is of the form $x_{n}=\beta^{-1}x_{n-1}e^{-\gamma\,x_{n-1}}$, where $\beta,\gamma>0$ are the model parameters. While $\beta^{-1}$ should be interpreted as the per capita reproduction rate, the term $e^{-\gamma\,x_{n-1}}$ takes care of the second requirement to prevent the population from unlimited growth due to limited resources. Environmental stochasticity may be introduced by allowing $\beta$ and/or $\gamma$ to vary in time.
The following stochastic version in which these parameters are replaced with iid $\R_{>}\times\R_{\ge}$-valued random variables $(\beta_{n},\gamma_{n})$ has been studied by Fagerholm and H\"ogn\"as \cite{FagerholmHog:02} (see also \cite{GylHogKos:94b,GylHogKos:94a} by Gyllenberg et al. for the case when only one parameter is random) and by Athreya \cite{Athreya:03} within a class of more general IFS on $\R_{\ge}$. For $n\ge 1$, consider the Markov chain (and IFS)
\begin{equation*}
X_{n}\ =\ \Phi_{n}(X_{n-1})\ :=\ \frac{X_{n-1}}{\beta_{n}}e^{-\gamma_{n}X_{n-1}},\quad n\ge 1,
\end{equation*}
with state space $\X=\R_{\ge}$. As in the previous example, it has 0 as an absorbing state and thus $\delta_{0}$ as a trivial stationary distribution. On the other hand, the IFS may also be studied on the positive halfline $\R_{>}$ in which case 0 can only be reached in the limit. As shown in \cite{Athreya:03}, a stationary distribution $\pi$ with $\pi(\R_{>})=1$ exists if
\begin{equation}\label{eq:extra moment Ricker}
-\infty<\Erw\log\beta<0,\quad\Erw\gamma<\infty\quad\text{and}\quad\Erw\rho^{-1}<\infty,
\end{equation}
where as usual $(\beta,\gamma)$ denotes a generic copy of the $(\beta_{n},\gamma_{n})$ and $\rho:=\beta\gamma$. By studying $(\log X_{n})_{n\ge 0}$ within the framework of Harris chains, a similar result has been obtained in \cite{FagerholmHog:02} under more restrictive assumptions (but allowing $\Prob(\gamma=0)>0$ which is obviously ruled out by \eqref{eq:extra moment Ricker}).

In order to study the behavior of $\pi$ at 0 or, equivalently, the tail of $\wh{\pi}$ (as defined in \eqref{eq:def pihat}) at $\infty$, we again consider the conjugation of $(X_{n})_{n\ge 0}$ with $x\mapsto x^{-1}$, viz.
\begin{equation*}
W_{n}\ =\ \Psi(W_{n-1})\ :=\ \beta_{n}W_{n-1}e^{\gamma_{n}/W_{n-1}},\quad n\ge 1,
\end{equation*}
on $\R_{>}$. Note that $\Psi$ is convex and attains its minimal value $\beta\gamma\,e$ at $x=\gamma$. We are therefore in a very similar situation as in the previous example, and the role of $4\xi$ is here taken by $\rho=\beta\gamma$. Again, we first give a result in the simpler case when $\rho$ stays bounded away from 0.

\begin{Theorem}\label{thm:tail Ricker model}
Assume \eqref{eq:extra moment Ricker} and that $\beta$ satisfies (IRT-1)-(IRT-3) for some $\kappa>0$. Then the following assertions hold true for any stationary distribution $\pi$ of $(X_{n})_{n\ge 0}$ on $\R_{>}$:
\begin{description}[(b)]
\item[(a)] $\wh{\pi}$ has upper tail index $\kappa$ and \eqref{eq:upper tail index pihat} holds.\vspace{.1cm}
\item[(b)] If $\rho\ge a/e$ a.s. for some $a>0$, and $\Erw\rho^{\kappa}<\infty$, then $\pi((0,1/a])=1$ and
\begin{equation}\label{eq:exact tail index pihat Ricker}
\lim_{x\downarrow 0}\,x^{-\kappa}\pi((0,x))\ =\ \frac{\Erw(\beta/X)^{\kappa}\big(e^{\kappa\gamma\,X}-1\big)}{\kappa\mu_{\kappa}}\ \in\ \R_{>},
\end{equation}
where $X$ is independent of $\Phi$ with $X\eqdist\pi$ and $\mu_{\kappa}=\Erw\beta^{\kappa}\log\beta$. In particular, $\wh{\pi}$ has exact tail index $\kappa$.
\end{description}
\end{Theorem}

\begin{proof}
The arguments are very similar to those in the proof of Thm.~\ref{thm:tail logistic transform} and therefore provided in shorter form.

\vspace{.1cm}
(a) First note that
\begin{align}\label{eq:Ricker lower bound}
\Psi(x)\ \ge\ F(x)\ :=\ \rho e\vee\beta x,\quad x\in\R_{>}.
\end{align}
The IFS generated by $F$ is mean contractive with stationary law $\pi_{F}$ satisfying
\eqref{eq:pi_F tail} by Prop.~\ref{prop:tails of max-type SFPE} (with $r=0$). This entails \eqref{eq:upper tail index pihat}.

\vspace{.1cm}
(b) If $\rho e\ge a$ a.s. for some $a>0$, then $\Psi(x)\ge\Psi(\gamma)=\rho e$ for all $x>0$ implies that $[a,\infty)$ is an absorbing set for the IFS $(W_{n})_{n\ge 0}$ generated by $\Psi$. But for $x\ge a$ and by bounding $\Psi(x)$ on $[a,r\gamma]$ and $[r\gamma,\infty)$ separately, we readily find that
$$ \Psi(x)\ \le\ G_{r}(x)\ :=\ r\rho e^{1/r}+\beta e^{1/r}x $$
for any $r>0$ so large that $\max_{a\le x\le r\gamma}\Psi(x)=\Psi(r\gamma)=r\rho e^{1/r}\ge\frac{r}{e}e^{1/r}\ge\Psi(a)$. The IFS generated by $G_{r}$ is mean contractive and its stationary distribution $\pi_{G}$ has exact tail index $\kappa_{r}<\kappa$ because $\Erw\rho^{\kappa}<\infty$ and $\beta e^{1/r}$ satisfies (IRT-1)-(IRT-3) for $\kappa_{r}$ such that $\Erw(\beta e^{1/r})^{\kappa_{r}}=1$. It follows by Thm.~\ref{thm:IFS basic tail result} that the lower tail index of $\wh{\pi}$ is bounded from below by $\kappa_{r}$ for all sufficiently large $r$. Since $\lim_{r\to\infty}\kappa_{r}=\kappa$ and by (a), we see that $\wh{\pi}$ has in fact tail index $\kappa$, in particular $\Erw W^{p}<\infty$ for any $p<\kappa$. 

\vspace{.08cm}
Finally, observe that the expectation in \eqref{eq:exact tail index pihat Ricker} equals the expectation of the positive random variable $\Psi(W)^{\kappa}-(\beta W)^{\kappa}=(\beta W)^{\kappa}(e^{\kappa\gamma/W}-1)$ and is therefore positive as well. It is also finite because (with $r$ so large that $e^{\kappa/x}-1\le\frac{2\kappa}{x}$ on $[r,\infty)$)
\begin{align*}
&\Psi(W)^{\kappa}-(\beta W)^{\kappa}\ \le\ 
\begin{cases}
\Psi(r\gamma)^{\kappa}+(r\rho)^{\kappa},&\text{if }a\le W\le r\gamma,\\
\hfill 2\kappa\beta^{\kappa}\gamma\,W^{\kappa-1},&\text{if }W\ge r\gamma,
\end{cases}\\
&\Erw\big(\Psi(r\gamma)^{\kappa}+(r\rho)^{\kappa}\big)\ \le\ c_1\,\Erw\rho^{\kappa}\ <\ \infty,
\intertext{and}
\Erw\beta^{\kappa}\gamma\,W^{\kappa-1}\ &=\ \Erw\beta^{\kappa-1}\rho\,\Erw W^{\kappa-1}\ \le\ \big(\Erw\beta^{\kappa}\big)^{(\kappa-1)/\kappa}\big(\Erw\rho^{\kappa}\big)^{1/\kappa}\Erw W^{\kappa-1}\ <\ \infty
\end{align*}
where H\"older's inequality and $\Erw W^{\kappa-1}<\infty$ have been utilized in the last estimation and where $c_{1},c_{2}\in\R_{>}$ denote suitable constants. Now assertion \eqref{eq:exact tail index pihat Ricker} follows by an appeal to Prop.~\ref{prop:IRT}.\qed
\end{proof}

The general case when $\Prob(\rho<a)>0$ for all $a>0$ is more complicated than the corresponding case in the previous example although it may be approached in a similar manner. We confine ourselves to a short discussion. First, pick $a>0$ such that $\Prob(\rho\ge a/e)>0$ and let $\sigma=\sigma_{1},\sigma_{2},...$ denote the successive epochs $n$ at which $\rho_{n}:=\beta_{n}\gamma_{n}\ge a/e$. Then the IFS $(W_{n}^{*})_{n\ge 0}$ generated by $\Psi_{\sigma:1}$ has state space $[a,\infty)$ (see \eqref{eq:Ricker lower bound}) and a stationary law $\wh{\pi}^{*}$ satisfying
$$ \wh{\pi}^{*}((t,\infty))\ \ge\ c\,\wh{\pi}((bt,\infty)) $$
for all $t>0$ and suitable $b,c>0$. Indeed, arguing as in Lemma \ref{lem:RLT geom sampling}, we find here that
\begin{align*}
\wh{\pi}^{*}((t,\infty))\ &=\ \frac{\Prob_{\pi}(\Psi_{1}(W_{0})>t,\,\rho_{1}\ge a/e)}{\Prob_{\pi}(\rho_{1}\ge a/e}\\
&=\ \frac{\Prob_{\pi}(\beta_{1}W_{0}e^{\gamma_{1}/W_{0}}>t,\,\rho_{1}\ge a/e)}{\Prob_{\pi}(\rho_{1}\ge a/e)}\\
&\ge\ \frac{\Prob_{\pi}\left(\frac{aW_{0}}{e\gamma_{1}}>t,\,\rho_{1}\ge a/e,\,e\gamma_{1}\le ab\right)}{\Prob_{\pi}(\rho_{1}\ge a/e)}\\
&\ge\ \frac{\Prob_{\pi}(W_{0}>bt)\,\Prob(\rho_{1}\ge a/e,\,e\gamma_{1}\le ab)}{\Prob_{\pi}(\rho_{1}\ge a/e)}\\
&=\ c\,\Prob_{\pi}(W_{0}>bt)\\ 
&=\ c\,\wh{\pi}((bt,\infty)),
\end{align*}
where $b$ is chosen sufficiently large and $c$ has the obvious meaning.

In view of Thm.~\ref{thm:tail Ricker model}(a) we are thus left with a proof of $\wh{\pi}^{*}$ having lower tail index $\kappa$. As an analog to what is shown in the proof of Lemma \ref{lem:tail of Wsigma} one can here verify as well that $F(x)\le\Psi_{\sigma:1}(x)\le G(x)$ with random functions $F(x)=\Pi_{\sigma}x\vee Q_{1}$ and $G(x)=\Pi_{\sigma}x+Q_{2}$ for $x>0$, where $\Pi_{n}:=\beta_{1}\cdot...\cdot\beta_{n}$ for $n\ge 1$,
\begin{align*}
Q_{1}\ :=\ \sum_{n=1}^{\sigma}\rho_{n}\quad\text{and}\quad 
Q_{2}\ :=\ a\sum_{n=1}^{\sigma-1}\beta_{\sigma}\cdot...\cdot\beta_{\sigma-n+1}\left(e^{\gamma_{\sigma-n+1}/\rho_{\sigma-n}}-1\right).
\end{align*}
Now, in order to formulate an analog of Thm.~\ref{thm2:tail logistic transform} for the present example, its conclusion being that $\wh{\pi}$ has exact tail index $\kappa$, we need assumptions on $(\beta,\gamma)$ (besides those in Thm.~\ref{thm:tail Ricker model}) which ensure $\Erw Q_{1}^{\kappa}<\infty$ and $\Erw Q_{2}^{\kappa}<\infty$.
While $\Erw Q_{1}^{\kappa}<\infty$ is easily seen to follow from $\Erw\rho^{\kappa}<\infty$, a natural sufficient condition on $(\beta,\gamma)$ for $\Erw Q_{2}^{\kappa}<\infty$ appears to be more difficult to find and will not be further discussed here.

\subsection{A class of random Lipschitz maps on $\R^{m}$}\label{subsec:Mirek}

We next take a brief look at an example in the multidimensional case, namely the class of IFS studied by Mirek in \cite{Mirek:11}. For a vector $x\in\R^{m}$, let $|x|$ denote its Euclidean norm. Put $|A|:=\max_{|x|=1}|Ax|$ for a $m\times m$ matrix $A$.
Consider a sequence $\Psi_{1},\Psi_{2},...$ of iid Lipschitz maps on $\R^{m}$ with generic copy $\Psi$ satisfying the following condition (see (H2) in \cite{Mirek:11}):\\
There exist a random variable $Q$ with $\Prob(Q>0)>0$, a positive random variable $\beta$, and a random $m\times m$ matrix $\Gamma$ taking values in a closed subgroup of the orthogonal group $\cO(\R^{m})$, such that
\begin{equation}\label{eq:Mirek's (H2)}
\sup_{x\in\R^{m}}|\Psi(x)-\beta\Gamma x|\ \le\ Q\quad\Prob\text{-a.s.}
\end{equation}
As an immediate consequence, note that
$$ \lim_{r\to\infty}r^{-1}\,\Psi(rx)\ =\ \beta\Gamma x\quad\Prob\text{-a.s.} $$
for all $x\in\R^{m}$.

\vspace{.1cm}
Under some natural additional assumptions, which particularly ensure that the IFS associated with $\Psi,\Psi_{1},...$ is contractive, the following discussion will show that its unique stationary law $\pi$ has exact tail index $\kappa>0$ in the sense that
\begin{equation}\label{eq:exact tail index Mirek}
\lim_{r\to\infty}\,r^{\kappa}\,\Prob_{\pi}(|X_{0}|>r)\ =\ \lim_{r\to\infty}\,r^{\kappa}\,\pi(\ovl{B}(0,r)^{c})\ =:\ C\ \in\ \R_{>},
\end{equation}
where $\ovl{B}(0,r)=\{x\in\R^{m}:|x|\le r\}$ and
$$ C\ =\ \frac{\Erw_{\pi}\left(|\Psi(X_{0})|^{\kappa}-\beta|X_{0}|^{\kappa}\right)}{\kappa\mu_{\kappa}},\quad\mu_{\kappa}\ :=\ \Erw\beta^{\kappa}\log\beta. $$
(see also \cite[(1.10) of Thm. 1.8]{Mirek:11}).

\vspace{.1cm}
Put $M:=|\beta\Gamma|$ and note that $M=\beta|\Gamma|=\beta>0$. As in \cite{Mirek:11}, we assume that $M$ satisfies (IRT-1)-(IRT-3) of Prop.\ \ref{prop:IRT} for some $\kappa>0$ and that $\Erw Q^{\kappa}<\infty$. Notice that \eqref{eq:Mirek's (H2)} provides us with
\begin{equation}\label{eq:(H2) reformulated}
(\beta|x|-Q)^{+}\ \le\ |\Psi(x)|\ \le\ \beta|x|+Q\quad\Prob\text{-a.s.}
\end{equation}
for all $x\in\R^{m}$. Under the stated assumptions, the IFS generated by $G(r)=\beta r+Q$ is 
contractive on $\R_{\ge}$ with stationary distribution $\pi_{G}$ having exact tail index $\kappa$ (Prop. \ref{prop:tails of perpetuities} and subsequent remark). By an appeal to Thm.\ \ref{thm:IFS basic tail result}, notably the right inequality in \eqref{eq:general iteration tail inequality}, we thus find that
$$ \limsup_{r\to\infty}\,r^{\kappa}\,\Prob_{\pi}(|X_{0}|>r)\ \le\ \lim_{r\to\infty}r^{\kappa}\,\pi_{G}((r,\infty))\ \in\ \R_{>}, $$
in particular $\Erw_{\pi}|X_{0}|^{p}<\infty$ for any $p\in (0,\kappa)$. The latter in combination with \eqref{eq:(H2) reformulated} may further be used to verify that $\Erw_{\pi}\left(|\Psi(X_{0})|^{\kappa}-\beta|X_{0}|^{\kappa}\right)<\infty$.

\vspace{.1cm}
On the other hand, the contractive IFS generated by $(\beta r-Q)^{+}$ has trivial stationary law $\delta_{0}$ and is therefore useless for completing the proof of \eqref{eq:exact tail index Mirek}. In fact, as also pointed out in \cite{Mirek:11} and easily sustained by the example $\Psi(x)=\beta\Gamma x$, the conditions imposed so far do not exclude the possibility that $\pi$ has bounded support. We close this discussion by pointing out that \eqref{eq:exact tail index Mirek} does indeed follow if the lower bound in \eqref{eq:(H2) reformulated} may be sharpened to
$$ |\Psi(x)|\ \ge\ (\beta|x|+Q')^{+}\quad\Prob\text{-a.s.} $$ 
for some random variable $Q'$ satisfying $\Prob(Q'>0)>0$ and $\Erw|Q'|^{\kappa}<\infty$. Just note that, by Prop. \ref{prop:tails of max-type SFPE}, the IFS generated by $F(r)=(\beta r+Q')^{+}$ has stationary law $\pi_{F}$ with exact tail index $\kappa$. Consequently, by another use of Thm.\ \ref{thm:IFS basic tail result}, we then infer
$$ \liminf_{r\to\infty}\,r^{\kappa}\,\Prob_{\pi}(|X_{0}|>r)\ \ge\ \lim_{r\to\infty}r^{\kappa}\,\pi_{F}((r,\infty))\ >\ 0 $$
and thereupon \eqref{eq:exact tail index Mirek} with $C\in\R_{>}$.

\subsection{A stable IFS of iid Lipschitz maps with more than one stationary law}\label{subsec:multiple law}

Let us finally briefly address the question of uniqueness of the stationary law for an IFS satisfying the conditions of Theorem \ref{thm:IFS basic tail result}. In the following, we provide a simple example of an IFS of iid Lipschitz maps on $\R_{\ge}$ with two unbounded disjoint invariant sets on which it is contractive (though naturally being noncontractive on the whole state space). By further verifying the conditions of Theorem \ref{thm:IFS basic tail result}, we then conclude that the two unique stationary laws on these sets and thus also any convex combination have the same exact tail index.

\vspace{.1cm}
Let $\Psi_{1},\Psi_{2},...$ be iid copies of the random Lipschitz map $\Psi:\R_{\ge}\to\R_{\ge}$, defined by $\Psi(x)=\alpha x+\beta$, where $\alpha$ takes values in $\{\frac{1}{3},2\}$ and has mean one, and
\begin{align*}
\beta\ =\ 
\begin{cases}
\hfill 3,&\text{if }x=m3^{n}\text{ for some }(m,n)\in\N_{0}\times\Z,\\
\hfill \gamma,&\text{otherwise}
\end{cases}
\end{align*}
with some standard exponential random variable $\gamma$ independent of $\alpha$. As one can easily see, $\Psi(I)\subset I$ and $\Psi(I^{c})\subset I^{c}$ a.s. for $I=\N_{0}3^{\Z}:=\{m3^{n}:m\in\N_{0},\,n\in\Z\}$, plainly a countable dense subset of $\R_{\ge}$. Moreover, $\alpha$ satisfies (IRT-1)-(IRT-3) with $\kappa=1$, and the IFS generated by the $\Psi_{n}$ is contractive on each of $I$ and $I^{c}$ with unique stationary distributions $\pi_{1},\pi_{2}$. Now observe that
$$ \alpha x+(3\wedge\gamma)\ =:\ F(x)\ \le\ \Psi(x)\ \le\ G(x)\ :=\ \alpha x+(3\vee\gamma) $$
for all $x\in\R_{\ge}$, and that the IFS generated by iid copies of $F$ and $G$, respectively, are contractive with unique stationary distributions having the same exact tail index, namely one. This follows once again by the result stated in \ref{subsec:perp}. Further details can be omitted. So we see that there are IFS with multiple stationary distributions to which our results apply, the conclusion being that all stationary laws must have the same tail index.

\bibliographystyle{abbrv}
\bibliography{StoPro}

\end{document}